\documentclass[12pt, leqno]{amsart}
\usepackage{indentfirst}
\usepackage{amssymb,amsmath}
\usepackage{amstext}
\usepackage{amsthm}
\usepackage{amsopn}
\usepackage{amsfonts}
\usepackage{amsmath}
\usepackage{latexsym}
\usepackage{amscd}
\usepackage{amssymb}
\usepackage{amsmath}
\usepackage[all,cmtip]{xy}
\usepackage{leftidx}
\usepackage{graphicx}
\usepackage{tikz}
\usepackage{ulem}

=\oddsidemargin \font\teneufm=eufm10 \font\seveneufm=eufm7
\font\fiveeufm=eufm5
\newfam\eufmfam
\textfont\eufmfam=\teneufm \scriptfont\eufmfam=\seveneufm
\scriptscriptfont\eufmfam=\fiveeufm
\def\frak#1{{\fam\eufmfam\relax#1}}
\let\goth\mathfrak
\def\cR{\mathcal R}

\def\gg{\goth g}

\def\GG{\frak G}

\def\gg{\goth g}
\def\gh{\goth h}
\def\gn{\goth n}
\def\gt{\goth t}
\def\gG { \goth {G}}

\def\gH{\goth H}

\def\gX{\text{\it X}}
\def\gT{\goth T}
\def\gB{\goth B}

\def\gb{\goth b}

\def\gg{\goth g}

\def\gz{\goth z}

\def\ga{\goth a}

\def\1{\mbox{\bf 1}}

 \DeclareMathOperator{\Hom}{Hom}
\DeclareMathOperator{\Aut}{Aut} 
 \DeclareMathOperator{\Der}{Der}
\DeclareMathOperator{\im}{im} 

 \DeclareMathOperator{\Id}{Id}

\newtheorem{theorem}{Theorem}[section]

\newtheorem{corollary}[theorem]{Corollary}

\newtheorem{lemma}[theorem]{Lemma}

\newtheorem{proposition}[theorem]{Proposition}

\theoremstyle{definition}
\newtheorem{remark}[theorem]{Remark}

\newtheorem{definition}[theorem]{Definition}

\numberwithin{equation}{section}

\def\Z{\Bbb Z}
\def\C{\Bbb C}

\def\tG{\widetilde{\mathbf G}}

\def\tE{\widetilde{E}}
\def\hE{\widehat{E}}

\def\C{\mathbb C}

\def\bB{\text{\rm \bf B}}

\def\bF{\text{\rm \bf F}}
\def\bJ{\text{\rm \bf J}}

\def\bG{\text{\rm \bf G}}
\def\bN{\text{\rm \bf N}}

\def\bW{\text{\rm \bf W}}
\def\Mor{\text{\rm Mor}}

\def\bZ{\text{\rm \bf Z}}

\def\pr{\prime}

\def\cL{{\mathcal L}}

\def\fet{\text{\it f\'et}}
\def\ad{\text{\rm ad}}

\def\bTor{\text{\rm \bf Tor}}

\def\bT{\text{\rm \bf T}}

\def\bS{\text{\rm \bf S}}

\def\bK{\text{\rm \bf K}}

\def\ts{{\tilde{\sigma}}}

\def\bmu{{\pmb\mu }}

\def\us{\underset}
\def\os{\overset}
\def\ol{\overline}
\def\id{\text{\rm id}}

\def\q{\quad}

\def\2int{\mathop{2\int}\nolimits}

\def\dim{\mathop{\rm dim}\nolimits}

\def\Spec{\mathop{\rm Spec}\nolimits}

\def\Hom{\mathop{\rm Hom}\nolimits}
\def\Der{\mathop{\rm Der}\nolimits}

\def\Gal{\mathop{\rm Gal}\nolimits}

\def\Pic{\mathop{\rm Pic}\nolimits}

\def\Aut{\text{\rm{Aut}}}

\def\bAut{\text{\bf{Aut}}}

\def\bOut{\text{\bf{Out}}}

\def\resp.{\mathop{\rm resp.}\nolimits}
\def\limproj{\mathop{\oalign{lim\cr
\hidewidth$\longleftarrow$\hidewidth\cr}}}
\def\limind{\mathop{\oalign{lim\cr
\hidewidth$\longrightarrow$\hidewidth\cr}}}

\def\lgr{\longrightarrow}

\font\math=cmmi10
\def\varpi{\hbox{\math\char'44}}
\def\simlgr{\buildrel\sim\over\lgr}

\def\pa{\S\kern.15em }

\def\un{\uppercase\expandafter{\romannumeral 1}}
\def\deux{\uppercase\expandafter{\romannumeral 2}}
\def\trois{\uppercase\expandafter{\romannumeral 3}}
\def\quatre{\uppercase\expandafter{\romannumeral 4}}
\def\cinq{\uppercase\expandafter{\romannumeral 5}}
\def\six{\uppercase\expandafter{\romannumeral 6}}

\def\gg{\goth g}

\def\hrf{\hbox to .2in{\hrulefill}}
\def\hfl#1#2#3{\smash{\mathop{\hbox to#3{\rightarrowfill}}\limits
^{\scriptstyle#1}_{\scriptstyle#2}}}
\def\gfl#1#2#3{\smash{\mathop{\hbox to#3{\leftarrowfill}}\limits
^{\scriptstyle#1}_{\scriptstyle#2}}}

\begin{document}

\title{Loop torsors. Theory and Applications}

\author{Vladimir Chernousov} 
\address{Department of Mathematics, University of Alberta,
    Edmonton, Alberta T6G 2G1, Canada}
\thanks{ V. Chernousov was partially supported by an NSERC research grant} 
\email{chernous@math.ualberta.ca}
\author[P. Gille]{Philippe Gille}
\thanks{P. Gille was supported by the project "Group schemes, root systems, and
related representations" founded by the European Union - NextGenerationEU through
Romania's National Recovery and Resilience Plan (PNRR) call no. PNRR-III-C9-2023-
I8, Project CF159/31.07.2023, and coordinated by the Ministry of Research, Innovation and Digitalization (MCID)
of Romania. }
\address{UMR 5208 Institut Camille Jordan - Universit\'e Claude Bernard Lyon 1
43 boulevard du 11 novembre 1918
69622 Villeurbanne cedex - France}  
\email{gille@math.univ-lyon1.fr}
\address{and Institute of Mathematics "Simion Stoilow" of the Romanian Academy,
21 Calea Grivitei Street, 010702 Bucharest, Romania}
\author{A. Pianzola}
\address{Department of Mathematics, University of Alberta,
    Edmonton, Alberta T6G 2G1, Canada.
    \newline
 \indent Centro de Altos Estudios en Ciencia Exactas, Avenida de Mayo 866, (1084) Buenos Aires, Argentina.}
\thanks{A. Pianzola wishes to thank NSERC and CONICET for their
continuous support}\email{a.pianzola@math.ualberta.ca}
\date{}
 \maketitle


 \maketitle 
 \begin{abstract} Loop torsors over Laurent polynomial rings in characteristic $0$ were originally introduced in relation to infinite dimensional Lie theory. Applications to other areas require a theory that can yields results in positive characteristic, and for group schemes that are not of finite type. The relation between loop and so-called toral torsors, is one of the central questions in the area. The present paper addresses  this question in full generality.
 \noindent  
 

\end{abstract}

\vskip14mm
\section{Introduction} Kac-Moody Lie algebras, the affine ones in particular, are defined by generators and relations. This begs the question of how do these algebras ``look like".\footnote{\,One encounters this question already in the finite dimensional Lie theory. That the split Lie algebra of type ${\rm G}_2$ exists is shown by considering generators and relations \`a la Serre. But this Lie algebra has a very concrete realization: The Lie algebra of derivations of the octonions.} The concept of loop torsor traces its roots to this question, that is, the characterization of affine Kac-Moody Lie algebras over the complex numbers. More precisely, the following classes of Lie algebras ``are the same":

(i) A Lie algebra $\cL$ obtained as the derived Lie algebra of an  affine Kac-Moody Lie algebra modulo its centre.\footnote{ \, The full algebra is obtained from $\cL$ by passing to  the universal central extension $\cL \oplus \Bbb{C}c$, and adding a canonical one dimensional space of ``degree derivations" that allows, among other things, for the existence of an invariant non-degenerate bilinear form.} 

(ii) A loop algebra $L(\gg, \sigma)$ where $\gg$ is a finite dimensional simple Lie algebra over $\C$ and $\sigma$ an automorphism of the Coxeter-Dynkin diagram of $\gg$ (viewed as an automorphism of $\gg.$)

(iii) A Lie algebra $L$ over the Laurent polynomial ring $R = \C[t^{\pm 1}]$ that locally  for the fppf-topology (in fact,  for the F\'et-topology) looks like $\gg \otimes_\C R$  for some (necessarily unique) $\gg$ as in (ii).

\smallskip
The loop algebras were independently introduced by Kac \cite{Kac}  and Moody (unpublished) as realization of the derived algebra of the affine Lie algebras with their centres factored out. From the definition it is not difficult to see that the infinite dimensional complex Lie algebra $L(\gg, \sigma)$ is in fact an $R-$Lie algebra, free of rank $= \dim(\gg)$, such that under the base change to $S = \C[t^{^{\pm \frac{1}{n}}} ]$ it becomes $\gg \otimes_\C S \simeq (\gg \otimes_\C R) \otimes_R S.$ It is thus a twisted form of $\gg \otimes_\C R$, hence corresponds to an $\bAut(\gg)-$torsor over $R.$ What (iii) says is that the torsors arising from loop algebras are in fact up to isomorphism {\it all}  $\bAut(\gg)-$torsors over $R$ (see \cite{P1} and \cite{P2}).

 The $\bAut(\gg)-$torsors corresponding to the loop algebras are very special. First of all, they are isotrivial, that is, they are trivialized by a finite \'etale extension of the base. They thus come from cocycles involving the algebraic fundamental group of $\Spec(R).$ The values of these cocycles are also special: They take values on $\bAut(\gg)(\C).$ In \cite{GP0} we introduced  the general concept of a loop torsor $E$ over a scheme $X$ under the action of an $X-$group scheme $\bG$ defined over a base field $k.$ We looked in detail at how the concept of $E$ being loop relates to the existence of maximal tori on the twisted $X-$group ${_E}\bG,$ in which case we say that $E$ is a toral torsor.

Loop torsors were used in \cite{CGP2} to provide a classification of  torsors over the ring $k[t_1^{\pm 1}, \cdots, t_n^{\pm 1}]$ ($k$ of characteristic $0$) under the action of a reductive group scheme. They were already singled out in \cite{PZ} and in the work of Stavrova \cite{St1}, \cite{St2}. Loop torsors seem to be central to the current work of R. Parimala and one of the authors, see \cite{GiPa}. Their framework, however,  requires that some of the key material of \cite{GP0}, such as Borel-Mostow semilinear considerations and the relation between loop and toral torsors,  be done over fields of arbitrary characteristic and for group schemes that are not of finite type. This is the purpose of the present work. We work with group schemes that include the group of automorphisms of reductive groups. They are extensions of twisted constant groups by reductive groups. The main results are that, with natural assumptions, toral torsor are loop (Theorem \ref{t->l}), and loop torsors are toral (Theorem \ref{existenceoftori}). The crucial application to the case in which the base is a Laurent polynomial ring is given in Theorem \ref{thm_laurent}.

Finally, it is our believe that the techniques introduced to establish this much broader base for the theory of loop torsors are of independent interest.

\section{Generalities on the algebraic fundamental group}\label{AFG}

Throughout this section $X$ will denote a scheme, and $\gG$ a group scheme over $X.$
We will do a quick review of the topics and concepts that are relevant for the present work. Accordingly, we assume that 
  $X$ is connected and  noetherian.  In particular $X$
  is quasi-compact and quasi-separated \cite[Tags 01OV,01OY]{SP}. 
Let $K$ be a field. Following \cite{LF}, by a {\it quasi geometric} point of $X$ we will understand a morphism $x : \Spec(K) \to X$ where $K$ is a separably closed field.

\medskip
\medskip

Let $X_{\fet}$ be the category of finite \'etale covers of $X,$ and
$F$ the covariant functor from $X_{\fet}$ to the category of finite sets
that assigns to a finite \'etale cover $X'$ of $X$ the set
$$
F(X^\pr) =\{\text{\rm quasi geometric points of}\; X^\pr(K)\;\text{\rm
above}\; x\}.
$$
That is, $F(X')$ consists of all morphisms $x^\pr : \;\Spec\,(K
) \to X^\pr$ for which the diagram
$$
\begin{matrix}
{} &{} &X^\pr\\
&\os {x^\pr}\nearrow &\downarrow\\
\Spec\, (K ) &\us x\rightarrow &X
\end{matrix}
$$
commutes. The group of automorphisms  of the functor $F$ is called
the {\it algebraic fundamental group of} $X$ {\it at} $x,$ and  is
denoted by $\pi_1(\gX,x).$\footnote{\, In Grothendieck's definition of $\pi_1(X,x),$ $x$ would be a geometric point. Just as in Galois theory, where separable closures instead of algebraic closures are  used, for the treatment of algebraic fundamental groups it is more natural/convenient to use quasi geometric base points.}
 If $X = \Spec(R)$ is affine, then $x$ corresponds to a ring homomorphism $R \to K,$  and we will denote the fundamental group by $\pi_1(R, x).$The functor $F$ is pro-representable: There exists a directed set
$I,$ objects $(X_i)_{i\in I} $ of $X_{\fet},$ surjective morphisms
$\varphi_{ij} \in \,\Hom_X(X_j,X_i)$ for $i\le j$ and geometric
points $x_i\in F(X_i)$ such that
$$
  x_i =\varphi  _{ij} \circ
x_j
$$
\begin{equation} \label{FG2}
 \text{\it The canonical map}\;  f:\limind\; \Hom_X(X_i, X^\pr)\to F(X^\pr) \,\, \text{\it is bijective,}
\end{equation}
where the map $f$ of (\ref{FG2}) is as follows:  Given
$\varphi :X_i\to X^\pr$ then $f(\varphi  ) = \varphi(x_i).$
 The elements of
$\limind\; \Hom_X(X_i,X^\pr)$ appearing in (\ref{FG2}) are by
definition the morphisms  in the category of pro-objects over $X$
(see \cite[\S8.13]{EGAIV} for details). It is in this sense that
$\limind\;\Hom(X_i,-)$ pro-represents $F.$

Since the $X_i$ are finite and \'etale over $X$ the morphisms
$\varphi_{ij}$ are affine. Thus the inverse limit
$$
X^{sc} = \limproj\;X_i
$$
exists in the category of schemes over $X$   \cite[\S8.2]{EGAIV}.  For any
finite \'etale scheme $X^\pr$ over $X$ we thus have a canonical map
$$
\Hom_{Pro-X}(X^{sc},X^\pr)  \buildrel {\rm def} \over = \limind\;\Hom_X(X_i,X^\pr) \simeq F(X') \to \;\Hom_X(X^{sc},X^\pr)
$$
obtained by considering the canonical morphisms $\varphi_i: X^{\rm
sc}\to X_i.$ Because our $X$ is assumed to be noetherian, this last map is bijective. More precisely.

\begin{proposition}\label{representability}  $F$ is represented
by $X^{sc};$ that is, there exists a
 bijection
$$
F(X^\pr) \simeq \Hom_X(X^{sc},X^\pr)
$$
which is functorial on the objects $X^\pr$ of $X_{\fet}.$
\end{proposition}

\begin{proof} Because the $X_i$ are affine over $X$ and $X$ is noetherian,
each $X_i$ is noetherian, in particular, quasicompact and
quasi-separated. 
 Thus, for $X^\pr/X$ locally of finite presentation, in particular for $X'$ in $X_{\fet}$,  the map
 $$ \limind\;\Hom_X(X_i,X^\pr) \to \Hom_X(\limproj X_i,X^\pr)$$
 is bijective \cite[prop 8.13.1]{EGAIV}. The
Proposition now follows from (\ref{FG2}).
\end{proof}


In computing $X^{ sc} = \limproj X_i$ we may replace $(X_i)_{i\in
I}$ by any cofinal family.  This allows us to assume that the $X_i$
are (connected) Galois, i.e. the $X_i$ are connected and the (left)
action of $\text{\rm Aut}_X(X_i)$ on $F(X_i)$ is transitive. We then
have
$$
F(X_i) \simeq \;\Hom_{Pro-X}(X^{sc},X_i) \simeq
\;\Hom_X(X_i,X_i) = \text{\rm Aut}_X(X_i).$$

Thus $\pi_1(X,x)$ can be identified with the group
$ \limproj \,\text{\rm Aut} _X(X_i)^{opp}.$  Each $\text{\rm
Aut}_X(X_i)$ is finite, and this endows $\pi_1(X,x)$ with the
structure of a profinite topological group.

The group $\pi_1(X,x)$ acts on the right on $X^{sc}$ as the
inverse limit of the finite groups $\text{\rm Aut}_X(X_i).$
Thus, the group $\pi_1(X,x)$ acts on the left on each set $F(X^\pr)
=\;\Hom_{Pro-X}(X^{sc},X^\pr)$ for all $X^\pr\in X_{\fet}.$ This
action is continuous since the structure morphism $X^\pr\to X$
``factors  at the finite level", that is,   there exists a morphism $X_i\to
X^\pr$ of $X$--schemes for some $i \in I.$  If $u: X^\pr \to  X^{\pr\pr}$ is a morphism of
$X_{\fet},$ then the map $F(u) :F(X^\pr) \to F(X^{\pr\pr})$ clearly
commutes with the action of $\pi_1(X,x).$ This construction provides
an equivalence between $X_{\fet}$ and the category of finite sets equipped with a continuous
$\pi_1(X,x)$--action.

The right action of $
\pi_1(\gX,x)$ on $X^{sc}$ induces an action of $\pi_1(X,x)$ on
$\gG(X^{sc}) = \;\Mor_X(X^{sc},\gG),$ namely
$$
^\gamma  f(z) = f(z^\gamma  )  \q \forall \gamma \in \pi_1(X,x), \q f\in
\gG(X^{sc}), \q z\in X^{sc}.
$$
\begin{remark}\label{scprop} Since  $X$ is connected, quasicompact and quasi-separated, the scheme $X^{sc}$ is connected and simply connected. See \cite[prop.\ 3.4]{VW}.
\end{remark}

\begin{proposition}\label{discrete} Assume that 
 $\gG$ is locally of finite
presentation over $X.$ Then ${\gG}(X^{sc})$ is a discrete
$\pi_1(X,x)$--group and the canonical map
$$
\us\longrightarrow{\text{\rm lim}}\,H^1\big(\text{\rm
Aut}_X(X_i),\gG(X_i)\big) \to H^1\big(\pi_1(X,x),\gG(X^{sc})\big)
$$
is bijective.
\end{proposition}

\begin{remark}\label{continuous} Here and elsewhere when a profinite group $A$ acts discretely on a module $M$ the corresponding cohomology $ H^1(A,M) $ is the {\it continuous} cohomology as defined in \cite{Se}. Similarly, if a group $H$ acts in both $A$ and $M,$ then $\Hom_H(A,M)$ stands for the continuous group homomorphism of $A$ into $M$ that commute with the action of $H.$
\end{remark}

\begin{proof}
To show that $\gG(X^{sc})$ is discrete amounts to showing that
the stabilizer in $\pi_1(X,x)$ of a point of $f\in \gG(X^{sc})$
is open. But if $\gG$ is locally of finite presentation then
$\gG(X^{sc}) = \gG(\limproj\,X_i) = \limind\, \gG(X_i)$
(\cite[prop. 8.13.1]{EGAIV} ), so we may assume that $f\in \gG(X_i)$
for some $i.$ The result is then clear, for the stabilizer we are
after is the inverse image
 under the continuous map $\pi_1(X,x) \to \text{\rm
Aut}_X(X_i)$  of the stabilizer of $f$ in $\text{\rm Aut}_X(X_i)$
(which is then open  since
  $\text{\rm Aut}_X(X_i)$ is given the
discrete topology).

By definition
$$
H^1\big(\pi_1(X,x),\gG(X^{sc})\big) =\limind
\big(\pi_1(X,x)/U, \gG(X^{sc})^U\big)
$$
where the limit is taken over all open normal subgroups $U$ of
$\pi_1(X,x).$ Since $I$ is directed, for each such $U$ we can find $U_i\subset U$ so
that $ U_i = \;\text{\rm ker}\big(\pi_1(X,x) \to \,\text{\rm
Aut}_X(X_i)\big).$ Thus
$$
H^1\big(\pi_1(X,x),\gG(X^{sc})\big) = \limind\;H^1\big(\text{\rm Aut}_X(X_i),\gG(X_i)\big)
$$
as desired.
\end{proof}

We assume now that $X$  is a $k$-scheme over a field $k$ such that $X(k) \neq \emptyset$
 and such that $X$ is geometrically connected.
We fix an algebraic closure $\bar{k}$ of $k,$ and let $k_s$ the separable closure of $k$ in $\overline{k}.$ 
\medskip

Fix an element $a_0 \in X(k)$ and denote by $a : \Spec (k_s) \longrightarrow \Spec(k) \os{a_0} \longrightarrow X$  the corresponding
quasi geometric point  of $X.$ 
The point $a$ will be our choice of base point in the sequel.

\begin{remark}\label{ks} (a) The scheme $X^{sc}$ has a natural $k_s-$scheme structure, as we now outline. Let $\ell$ be a finite algebraic extension of $k,$ and denote $X \times_k \, \ell$ by $X_\ell.$ There exists a unique quasi geometric point $a_\ell : \Spec(\ol{k}) \to X_\ell$ such that $a_\ell \mapsto a$ (resp. $a_\ell \mapsto \Spec(\ell)$ ) under the projection $X_\ell \to X$ (resp. $X_\ell \to \Spec(\ell)$ ). Assume now $\ell$ is Galois. Since $X$ is geometrically connected, $X_\ell$ is a (connected) Galois extension of $X$ with Galois group canonically isomorphic to that of $\ell/k.$ 
 
The set $I_\ell = \{i \in I :  \Hom_X(X_i, X_\ell) \neq \emptyset\}$ is not empty. If $i \in I_\ell$ and $f_i \in \Hom_X(X_i, X_\ell),$ then $f_i$ is surjective and, after following $f_i$ with an element of $\Aut_X(X_\ell)$ if necessary, we may assume that $f_i(a_i) = a_\ell.$ Such an $f_i$ is unique.\footnote{\, Because $X_\ell/X$ is finite \'etale and $X_i$ is connected, $f_i$ is determined by its value at any geometric point of $X_i.$ Indeed, let $f,g \in \Hom_X(X_i, X_\ell)$ be such that $f(a_i) = g(a_i) = a_\ell.$  Let $p_i : X^{sc} \to X_i$ be the canonical projection. Then if $f' = f \circ p_i$ and $g' = g \circ p_i,$ we have  $f'(a) = g'(a).$ From the bijection $\Hom_X(X^{sc}, X_\ell) \simeq F(X_\ell)$ we conclude that $f' = g'.$ Since $p_i$ is surjective we conclude that $f = g.$} If $i < j$ and $\varphi_{ij} : X_j \to X_i$ is one of our transition morphisms, then
 $f_i \circ \varphi_{ij} \in \Hom_X(X_j, X_\ell),$ 
 so $j \in I_\ell.$  By uniqueness $f_j = f_i \circ \varphi_{ji}.$
 
 After applying the projection $X_\ell \to \Spec(\ell)$ to the above considerations we obtain the existence of $p_\ell \in \Hom_k\big(X^{sc}, \Spec(\ell)\big)$ such that $p_\ell(a) = a_\ell.$ This induces our desired  morphism $p_s : X^{sc} \to \limproj \Spec(\ell) = \Spec(k_s).$
 
  By \cite[Prop. 3.2.14]{LF} $\pi_1\big(\Spec(k), a\big) \simeq \Gal(k) = \Gal(k_s/k).$ The structure morphism $X \to \Spec(k)$ induces a group homomorphism  $p : \pi_1(X,a) \to \Gal(k_s/k).$ One checks from the explicit construction given above that the structure morphism  $p_s : X^{sc} \to \Spec(k_s)$ is $\pi_1(X,a)-$equivariant. 
 
 (b) We claim that $p_s$ has a section. The geometric point $a_i$ corresponds to a point $x_i$ of the underlying topological space of $X_i.$ The residue field $\kappa(x_i)$ will for convenience be denoted by $k_i.$ We thus have a $k-$scheme morphism $\psi_i : \Spec(k_i) \to X_i.$ 
 A transition morphism $\varphi_{ij} : X_j \to X_i$ yields the local ring morphism $\varphi_{ji}^* :{\mathcal{O}}_{X_i, x_i} \to {\mathcal{O}}_{X_j, x_j}.$ This yields a $k-$field extension $k_i \to k_j,$ hence a $k-$scheme morphism $\psi_{ji}: \Spec(k_j) \to \Spec(k_i).$ 
 
 Since $X_i$ is a finite \'etale extension of $X$ and $x_i \mapsto x,$ $k_i$ is a finite separable extension of $k.$ Thus $\limproj \Spec(k_i) = \Spec(k_s).$
Consider the projection $\pi_i : \Spec(k_s) \to \Spec(k_i),$ and let $g_i = \psi_i \circ \pi_i :\Spec(k_s) \to X_i.$ From the above discussion we see that $\varphi_{ij} \circ g_j = g_i.$ We thus get an induced scheme morphism $a_s : \Spec(k_s) \to \limproj X_i = X^{sc}.$ It is clear from the construction that $a_s$ followed by the structure morphism $p_s : X^{sc} \to \Spec(k_s)$ is the identity.

 \end{remark}

  We will denote
$X \times_k {k_s}$ by $X_{k_s}.$ 
Our point $a_0: \Spec(k) \to X$ gives rise by base change to 
a quasi geometric point $a_s: \Spec(k_s) \to X_{k_s}.$ By
\cite[Proposition 3.3.7]{LF} the sequence
\begin{equation}\label{fundamentalexact}
1 \to \pi_1(X_{k_s}, a_s) \to  \pi_1(X,  a) \buildrel p \over \lgr
\Gal(k) \to 1
\end{equation}
is exact, and in fact split exact. Indeed, the existence of the rational point $a_0 \in X(k)$ yields a natural choice of a section of $p$ that can be understood as follows. Our $a_0 : \Spec(k) \to X$ induces by functoriality a group homomorphism
  $q : \pi_1\big(\Spec(k), b\big) \to \pi_1(X,a)$ where $b : \Spec(k_s) \to \Spec(k)$ is the natural map. Because $a$ was also defined via $a_0$, i.e. $b$ is $a$ followed by the structure map $X \to \Spec(k),$ we see that $q \circ p = \Id.$ Finally $\pi_1\big(\Spec(k), b\big) \simeq \Gal(k_s/k).$ In concrete terms we are identifying a finite Galois extension $\ell$ of $k$ with Galois group $\Gamma$ with the Galois cover $X \times_k \ell$ of $X.$

\section{Loop and toral torsors}

Throughout $X$ denotes a scheme and $\gG$  a group scheme over $X.$
 
\subsection{Generalities about torsors} 

Let $E$ be an fppf sheaf on $X$ equipped with a right action of $\gG.$ Recall that $E$ is said to be a (right)  {\it $\gG-$sheaf torsor over
$X$ } --or simply a $\gG$--sheaf torsor if $X$ is understood-- if there exists a faithfully flat and locally finitely presented morphism $Y \to X$,
 such that $E \times_X Y \simeq \gG
\times_X Y = \gG_{Y},$ where $\gG_{Y}$ acts on itself by right translation.\footnote{\, We are viewing  the $Y$ scheme $\gG_Y$ as an fppf group sheaf on $Y.$} In the above situation we say that $Y$ {\it trivializes} $E.$ In what follows we denote by  $H^1(X,\gG)$ the pointed set of isomorphism classes of $\gG-$sheaf torsors over X. The set $H^1(X,\gG)$ can be computed in terms of cocycles \`a la \v{C}ech as we now briefly recall (see \cite[\S2]{Gi4} for details).

Given a base change $Y\to X$, we denote by  $H^1(Y/X,
\gG)$ the kernel of the base change map
$H^1(X,\gG) \to H^1(Y, {\gG}_{Y}).$  Thus $H^1(Y/X,
\gG)$ can be thought as the set of isomorphism classes of $\gG-$sheaf torsors over $X$ that are trivialized by the base change to $Y.$ As it is customary, and when no confusion is possible,  we will  denote in what follows $H^1(Y,\gG_{Y})$ simply by $H^1(Y,\gG).$

If the base change $Y/X$ is a an fppf cover of $X$, then $H^1(Y/X,
\gG)$ can be computed in terms of cocycles. More precisely $H^1(Y/X,
\gG) = \text{\v{H}}^1_{fppf}(Y/X, \gG).$ The ``set" of fppf covers of $X$ is  directed and by passing to the limit over all such coverings $Y$, we have a bijection of pointed sets
$$ 
\us\longrightarrow{\text{\rm lim}}\,H^1(Y/X,\gG) \to H^1(X,\gG).
$$

Recall that a sheaf torsor $E$ over $X$ under $\gG$ is called {\it
isotrivial} if it is trivialized by some {\it finite} \'etale
cover of $X,$ that is,
$$
[E] \in H^1(X^\pr/X,\gG) \subset H^1(X,\gG)
$$
for some (surjective) $X^\pr \to X$ in $ X_{\fet}.$ We denote by $H^1_{iso}(X,\gG)$ the 
subset of $H^1(X,\gG)$ consisting of classes of isotrivial torsors.

\begin{proposition}\label{discrete1} 
Assume that $X$  is  connected and noetherian, and that $\gG$ is locally of finite
presentation over $\gX.$  Then
$$
H^1_{iso}(X,\gG) = \;\text{\rm ker}\, \big(H^1(X,\gG) \to H^1(X^{sc},\gG)\big).
$$
\end{proposition}

\begin{proof}
Assume $E$ is trivialized by $X^\pr\in X_{\fet}.$  By (\ref{FG2}) there exists $i \in I$ such that $\Hom_X(X_i, X') \neq \emptyset.$ Then $E\times_X
X_i = E\times_X X^\pr \times_{X^\pr}X_i = \gG_{X^\pr}\times_{X^\pr}
X_i = \gG_{X_i}$ so that $E$ is trivialized by $X_i.$  The image of
$[E]$ on $H^1(X^{sc}, \gG)$ is thus trivial.
 
Conversely assume $[E]\in H^1(X,\gG)$ vanishes under the base change
$X^{sc}\to X.$  Since the $X_i$ are quasicompact and
quasiseparated, and $\gG$ is locally of finite presentation, a
theorem of Grothendieck-Margaux \cite{Mg1} shows that the canonical
map
$$
\limind \;\text{\v{H}}^1_{fppf}(X_i,\gG) \to \text{\v{H}}^1_{fppf}(X^{sc},\gG)
$$
is bijective. Thus $E\times_{X} X_i\simeq \gG_{X_i}$ for some $i\in
I.$
\end{proof}

The question of whether a given sheaf torsor $E$ is representable by an $X-$scheme is quite delicate. If this is the case, we say that $E$ is a {\it torsor}. Representability holds if  $\gG$ is ind quasi affine (see \cite[Tag 0AP6]{SP}). This covers all the cases we are interested in. For future use, however, it is important to formulate this general material, and the definition of loop torsors, at the level of sheaves.
\subsection{Toral torsors}\label{tt}
Let $\mathfrak{G}$ be an $X-$group which is locally of finite type. If $x \in X$ the fibre of $\mathfrak{G}_x$  is an algebraic group over $\kappa(x).$ Following \cite{SGA3}  we say that a closed subgroup $\gT$ of $\mathfrak{G}$ is a {\it maximal torus} if it is a torus (as a group scheme over $X$) and $\gT_x$ is a maximal torus of $\mathfrak{G}_x$ for all $x \in X.$

Let $[E] \in H^1(X, \gG).$ We can then consider the twisted $X-$group sheaf ${_E}\gG = E \wedge ^{\gG} \gG.$ We will henceforth assume that ${_E}\gG$ is representable, that is, a group scheme over $X$ (for example if $\gG$ is ind quasi affine over $X$). This $X-$group is fppf locally isomorphic to $\gG,$ hence locally of finite type. We say that $E$ is a {\it toral} $\gG-${\it torsor} if the $X-$group ${_E}\gG$ admits a maximal torus.\footnote{\, If $\gG = \bG_X$ for some algebraic $k-$group $\bG$ we recover  the definition of toral $\bG-$torsor given in \cite{GP0}.}


\subsection{Loop torsors}

Throughout this section $k$ will denote a
field, $\gX$ a geometrically connected noetherian scheme over $k$, and
$\gX^{sc} = \limproj\;\gX_i$ its simply connected cover of $X$ with respect to a geometric point $a$ of $X$ arising from a rational point $a_0 \in X(k)$ as
described in \S\ref{AFG}. 
By Remark \ref{ks} we have a structure morphism $p_s : X^{sc} \to \Spec (k_s).$ This yields a natural group homomorphism $p_s^{\bG} : \bG(k_s) \to \bG(X^{sc}).$ By Remark \ref{ks}, $p_s^{\bG}$ is $\pi_1(X,a)-$equivariant. More precisely, the group $\pi_1(X,a)$ acts on $k_s,$ hence on $\bG(k_s),$ via the
group homomorphism $\pi_1(X,a) \to\,\Gal \,(k)$ of
(\ref{fundamentalexact}). This action is continuous and together
with $p_s^{\bG}$ yields a map of pointed sets
$$
H^1\big(\pi_1(X,a),\bG(k_s)\big) \to
H^1\big(\pi_1(X,a),\bG(X^{sc})\big).\footnote{\, We remind the reader that these $H^1$ are defined in the ``continuous" sense (see Remark \ref{continuous}).}
$$
 On the other hand, by Proposition \ref{discrete} and basic
properties of torsors trivialized by Galois extensions, we have
$$
\begin{aligned}
H^1\big(\pi_1(\gX,a),\bG(\gX^{sc})\big)
&= \limind \;H^1\big(\text{\rm Aut}_\gX(\gX_i),\bG(\gX_i)\big)\\
&= \limind \;H^1(\gX_i/\gX,\bG)\subset H^1(\gX,\bG).
\end{aligned}
$$

By means of the foregoing observations we can make the following.
\begin{definition} \label{looptorsor} A $\bG-$sheaf torsor $E$
over $X$ is called a {\it loop sheaf torsor} if its
isomorphism class $[E]$ in $H^1(X,\bG)$ belongs to the image of the
composite map
$$
H^1\big(\pi_1(X,a),\bG(k_s)\big) \to
H^1\big(\pi_1(X,a),\bG(X^{sc})\big)\subset H^1(X,\bG).
$$
\end{definition}
If $E$ is representable, then we say that it is a {\it loop torsor.}

\medskip

 \subsection{Variations of the fundamental group}

For applications of loop torsors or sheaf torsors, the fundamental group $\pi_1(X,a)$ is 
 often too big
and it suffices to work with some of its quotients. Below we define two of the most useful cases.

As above, $k$ is a field  and $k_s$ the separable closure of  
$k$ in a fixed algebraic closure $\overline{k}$ of $k.$ We let $p$ be the characteristic exponent of $k.$ $\Gamma_k=\Gal(k_s/k)$ will denote the absolute Galois group of $k$.

For a profinite group $\mathcal{G}$, we denote by $\mathcal{G}^{(p')}$ its maximal prime to $p$
quotient: it is the quotient of $\mathcal{G}$ by the closure of the (normal) subgroup
$\mathcal{G}^{[p]}$ generated by the pro-$p$--Sylow subgroups of $\mathcal{G}$.
Applying this construction to $\pi_1(X,a)$ we then get the profinite group
$\pi_1( X, a)^{(p')}$. 

Another useful quotient of $\pi_1(X,a)$ is the following one.
Since $\pi_1( X_{k_s}, a)^{[p]}$ is normal and closed in $\pi_1( X, a)$,
we can define the quotient $\pi_1( X, a)^{(p'-geo)}= \pi_1( X, a) / \pi_1( X_{k_s}, a)^{[p]}$.
Altogether we have a commutative  diagram of exact sequences of 
profinite groups
\[
\xymatrix{
1 \ar[r]& \pi_1( X_{k_s}, a) \ar[r] \ar[d]& \pi_1( X, a) \ar[d] \ar[r]&\Gal(k_s/k) \ar[r]
\ar[d]^{=}& 1 \\
1 \ar[r]& \pi_1( X_{k_s}, a)^{(p')} \ar[d]^{=} \ar[r]& \pi_1( X, a)^{(p'-geo)}
\ar[d] \ar[r]&\Gal(k_s/k) \ar[d] \ar[r]& 1 \\
1 \ar[r]& \pi_1( X_{k_s}, a)^{(p')} \ar[r]& \pi_1( X, a)^{(p')} \ar[r]&\Gal(k_s/k)^{(p')} \ar[r]& 1 
}
\]
 By Galois theory, we have the associated  cover $X^{sc, p'-geo}$ (resp.\
 $X^{sc, p'}$) of $X$ called the universal geometrical $p'$-cover  (resp.\ 
 universal $p'$--cover). In particular $k_s$ admits the
 maximal $p'$-Galois subextension $k^{(p')}$.

In what follows a {\it Galois $p'-$geometric cover of $X$} means a finite Galois cover of $X$ whose Galois group arises as a quotient of $\pi_1(X,x)^{p'-geo}.$

\begin{definition} \label{looptorsor1} A $\bG-$(sheaf) torsor $E$
over $X$ is called a {\it $p'$-geometric loop torsor} if its
isomorphism class $[E]$ in $H^1(X,\bG)$ belongs to the image of the
composite map
$$
H^1\big(\pi_1(X,a)^{(p'-geo)},\bG(k_s)\big) \to
H^1\big(\pi_1(X,a),\bG(X^{sc})\big)\subset H^1(X,\bG).
$$
\end{definition}

\begin{definition} \label{looptorsor2} A $\bG-$(sheaf) torsor $E$
over $X$ is called a {\it $p'$-loop torsor} if its
isomorphism class $[E]$ in $H^1(X,\bG)$ belongs to the image of the
composite map
$$
H^1\big(\pi_1(X,a)^{(p')},\bG(k^{(p')})\big) \to
H^1\big(\pi_1(X,a),\bG(X^{sc})\big)\subset H^1(X,\bG).
$$ 
\end{definition}

According to the definitions we thus have the inclusions of pointed sets 
$$
H^1_{p'-loop}(X,G) \subseteq  H^1_{\substack{p'-geo-loop}}(X,G) \subseteq H^1_{loop}(X,G).
$$


\section{The loop nature of  toral torsors for extensions of locally constant $k-$groups by reductive groups}

Throughout this section $k$ is
a field whose characteristic exponent will be denoted by $p,$ and
$\gX$ a geometrically connected noetherian scheme over $k.$ 

Let us begin by explaining the relevance of considering $k-$groups obtained by extending a locally constant group by a reductive group.  If $\bG$ is a reductive $k-$group, then 
$ \bAut(\bG)$ is a smooth $k-$group that fits in the exact sequence
$$1 \to \bG_{ad} \to \bAut(\bG) \to \bOut(\bG) \to 1$$
where $\bG_{ad} = \bAut(\bG)^{\circ} = \bG/\mathcal{Z}(\bG)$ is the adjoint group of $\bG$, and $\bOut(\bG)$ is a twisted constant group. $\bAut(\bG)$ is an algebraic group if $\bG$ is semisimple, but in general it is only locally algebraic. For example, if $\bG$ is a two-dimensional split torus, then $\bAut(\bG) = \bOut(\bG)$ is the constant $k-$group corresponding to the (abstract) group ${\rm GL}_2(\Z).$  

Up to isomorphism, reductive $X-$groups which are twisted forms of $\bG_X$  are classified by $H^1\big(X, \bAut(\bG)\big).$  We may want to understand if a given reductive $X-$group is loop, or if it has a maximal torus. This involves an $X-$torsor $E$ under $\bAut(\bG)$. In \cite{G} and \cite{GiPa} it is shown that the right approach to the question is through an $H^1(X, \tG)$ where, roughly speaking, $\tG$ plays the role of $\bAut(\bG)$ with the adjoint group replaced by $\bG$ itself. This $H^1(X, \tG)$ keeps track of {\it both} $E$ and the twisted $X-$group ${_E}\bG.$
\medskip

\centerline{***}

In what follows we will deal with an exact sequence of $k-$groups
\footnote{\, Throughout exact sequence is meant as an exact sequence of fppf sheaves of the base scheme.}
\begin{equation}\label{mainsequence}
1 \to \bG \to \tG \to \bJ \to 1
\end{equation}
where 

(a) $\bG$ is reductive, and

(b) $\bJ$ is twisted constant.

\begin{lemma}\label{localg}
The $k-$group $\tG$ is locally algebraic, smooth and ind quasi affine.
\end{lemma} 
\begin{proof} We may assume without loss of generality that $k$ is algebraically closed.\footnote{\, That the property of being ind quasi affine is stable under base change is clear; that this property descends is a result of Gabber \cite[Tag 0APK]{SP}.} We claim that then $\tG$ is a disjoint union of copies of $\bG,$ in particular locally algebraic, \'etale and, by \cite[37.66.1]{SP}, ind quasi affine.

To establish the claim it will suffice to show that the map $\tG(k) \to \bJ(k)$ is surjective. Let $x \in \bJ(k).$ There exists a $k-$ring of finite type $R$ and an element $g_R \in \tG(R)$ such that $g_R$ maps to the image $x_R$ of $x$ in $\bJ(R).$
If $\mathfrak{m}$ is a maximal ideal of $R,$ then by functoriality we see that under the map $\tG(R) \to \bJ(R/\mathfrak{m}) = \bJ(k)$ the element $g_R$ maps to $x.$
\end{proof}

\begin{remark} It follows from the Lemma that all $\tG-$sheaf torsors over a $k-$scheme $X$ are representable, that is, they are torsors.
\end{remark}
Since  $\bG$ is quasicompact (being affine) the morphism $\bG \to \tG$ is a closed immersion (\cite[VI$_{\rm A}$ Prop. 2.5.2]{SGA3}). We will henceforth identify $\bG$ with a closed subgroup of $\tG.$

For details about the connected component of the identity $\tG ^{^\circ}$ of $\tG,$ see \cite[VI$_{\rm A,B}$]{SGA3}. Since $\bG$ is connected and $\bJ$ is a discrete topological space, we see that $\tG^{^\circ} = \bG.$ If $\bS$ is a $k-$subgroup of $\tG$ that is a torus, then $\bS$ is connected and therefore $\bS \subset \bG.$ Thus if $\bT$ is a maximal torus of $\bG,$ then $\bT$ is a maximal torus  of $\tG$ (see \S \ref{tt}). If follows that the $k-$functor $\bTor(\tG)$ of maximal tori of $\tG$ is a $k-$scheme isomorphic to $\bTor(\bG)$ . Moreover, $\bTor(\bG) \simeq \bG/\bN_{\bG}(\bT)$ (see \cite[ ${\rm XXII}$  \S5.8] {SGA3}).


Let $\tE$ be a $\tG-$torsor over $X$ and consider the reductive $X-$group  ${_{\tE}}\bG := {_{\tE}}(\bG_X).$
We fix  a maximal torus $\bT$ of $\bG$ and recall  that $\bN = \bN_{\tG}(\bT)$ is a closed $k$-subgroup of $\tG$ (see \cite[$\rm{VI}_B$ 6.2.5]{SGA3}), hence locally algebraic. Let $\tE$ be a $\tG-$torsor over $X$ and consider the $X-$group ${_{\tE}}\tG := {_{\tE}}(\tG_X) := \tE \wedge^{\tG} \tG.$\footnote{\,  A priori ${_{\tE}}\tG$ is a group sheaf over $X.$ That it is representable --i.e., that it is a group scheme over $X$-- follows from \cite[Lemma 2.7]{G} and \cite[Tag 0APK]{SP}, given that $\tG_X$ is an ind-affine scheme over $X$ (Lemma \ref{localg} and \cite[Tag 0AP7]{SP}).}  It is a twisted form (fppf) of $\tG,$ hence locally of finite presentation over $X.$
Since $\bG$ is a normal subgroup of $\tG$ we can also consider the $X-$group ${_{\tE}}\bG := {_{\tE}}(\bG_X) := \tE \wedge^{\tG_X} \bG_X.$ It is a twisted form of $\bG_X,$ hence reductive.

Assume for the moment that $[\tE] \in \im[H^1(X, \bN) \to H^1(X, \tG)].$ We can then consider the twisted $X-$torus ${_{\tE}}\bT  := \tE \wedge^{\bN_X} \bT_X.$ The main result of this section is the following.

\begin{theorem}\label{t->l} Assume that  $X = \Spec(A)$ is affine, integral, normal, has a rational point, i.e. $X(k) \neq \emptyset.$  Let $\tE$ be a $\tG$-torsor over $X$ such that 
 $_{\tE}\tG$ admits a maximal $X$--torus $\gT'$ (in other words $\tE$ is toral).
Let $Y=\Spec(B)$ be a  (connected) Galois 
 cover of $X$ satisfying the  following three conditions:

\smallskip

{\rm (a)} $\Pic(B)=0$;

\smallskip

{\rm (b)} $\bT_B$ is split;

\smallskip

{\rm (c)} $\tE \times_X Y$ is trivial.

\smallskip

\noindent
Then:

\smallskip

{\rm (i)} The $\tG_X$-torsor $\tE$ admits a reduction to an $\bN$-torsor $\hE$ such that $\hE$ is split by $Y/X$ and  $_{\hE}\bT$ is isomorphic to $\gT'$.

\smallskip

{\rm (ii)}   Assume furthermore that  the $X-$torus $\gT'$ is split by a finite Galois extension $X'$ of $X$ of $p'$-degree with trivial Picard group. Then $\tE$ is a loop $\tG-$torsor.

\smallskip

{\rm (iii)}   Assume furthermore that $Y/X$ is a $p'$-geometric cover and that the $X-$torus $\gT'$ is split by a finite Galois extension $X'$ of $X$ of $p'$-degree with trivial Picard group. Then $\tE$ is a $p'$-geometric loop $\tG-$torsor.




\end{theorem}

The proof, which will be given at the end of the section,  involves some preliminary results and considerations.

\subsection{Torsors under locally constant groups}

\begin{lemma}\label{K} Let $\bK$ be a twisted contant $k-$group. Then $\bK$ is \'etale and $\bK_s = \bK \times_k k_s$ is constant.
\end{lemma}
\begin{proof}  That $ \bK_s$ is constant is shown in \cite[ I \S4.6.2]{DG}.  A constant group is clearly \'etale. By fpqc descent $\bK$ is \'etale.
\end{proof}

\begin{lemma} \label{constantimpliesloop}  Let $\bJ$ be a twisted constant $k-$group  and $E$ a $\bJ-$torsor over X. If $E$ is isotrivial, then $E$ is a loop torsor.
\end{lemma}
\begin{proof} Let $\bJ_s = \bJ \times_k k_s.$ By the last Lemma $\bJ_s$  is constant. In particular $\bJ$ is locally of finite presentation (by fpqc descent). Observe that:
 \smallskip
 
(a) Since $[E]$ is isotrivial and $X$ is noetherian, we see by Proposition \ref{discrete1} that $[E]$ becomes trivial after passage to $X^{sc}.$ We can thus view $[E]$ as an element of $H^1\big(\pi_1(X,a),\bJ(X^{sc})\big).$

\smallskip

(b) The canonical map $ \bJ(k_s) \to \bJ(X^{sc})$ is bijective and, because of  (\ref{splitpi}), an isomorphism 
of $\pi_1(X)-$modules.
\smallskip

It follows from (a) and (b) that $[E] \in H^1(\pi_1(X), \bJ(k_s)).$ Thus $E$ is by definition a loop torsor.
\end{proof}

\begin{corollary}\label{tor-loop} Assume that $X$ is normal. Every $X-$torsor under  a twisted constant $k-$group $\bJ$ is loop. 
\end{corollary}
\begin{proof} Indeed, any such torsor is isotrivial by  \cite{G} Lemma 2.14.
\end{proof}

\begin{lemma}\label{Exactconstant} Let $1 \to \bF_0 \os{i}\longrightarrow \bF \os{h}\longrightarrow \bF_1$ be an exact sequence of locally algebraic $k-$groups. Assume that $\bF_0$ and $\bF_1$ are  twisted constant $k-$groups. Then $\bF$ is a twisted constant $k-$group.
\end{lemma}
\begin{proof} We may assume without loss of generality that $k$ is separably closed. Thus $\bF_0$ and $\bF_1$ are constant by Lemma \ref{K}. Let $\bF^\circ$ (resp. $\bF_0^\circ$, $\bF_1^\circ$) be the connected component of the identity of $\bF$ (resp. $\bF_0$, $\bF_1$), see \cite[${\rm VI}_A \,2.3$]{SGA3}. Since $\bF_1^\circ =\{e_{\bF_1}\}$ (\cite[II \S5.1.4]{DG}) and $h$ is continuous, we have $h(\bF^\circ) = \bF_1^\circ.$  It follows that $\bF^\circ$ is a $k-$subgroup of $\ker h.$ The morphism $j: \bF^\circ \to \ker h \simeq \bF_0$ is in fact a monomorphism (\cite[${\rm VI}_B$ 1.4.2]{SGA3}), hence injective on points.  But $j$ maps $\bF^\circ$ into $\bF_0^\circ = \{e_{\bF_0}\}.$ Thus  $\bF^\circ = \{e_\bF\}.$ 
By \cite[II \S5.1.4]{DG}  $\bF$ is \'etale, hence constant (\cite[I \S4.6.2]{DG}).
\end{proof}

\subsection{Proof of Theorem \ref{t->l}}


Let $\tE$ be a $\tG-$torsor over $X$ and consider the reductive $X-$group  ${_{\tE}}\bG := {_{\tE}}(\bG_X).$
Recall that $\bT$ is a maximal torus of $\bG$ which is fixed in our discussion.  
By means of the above considerations 
we view $\bT$ as a maximal torus of $\tG.$ As already shown the normalizer $\bN_{\tG} (\bT),$ which we denote by $\bN,$  is a closed $k-$subgroup of $\tG.$ 

\begin{proposition}\label{redGN}
Under the assumptions of Theorem \ref{t->l},
the $\tG_X$-torsor $\tE$ admits a reduction  to an $\bN$-torsor $\hE$ such that $\hE$ is split by $Y/X$ and 
 $_{\hE}\bT$ is isomorphic to $\gT'$.
\end{proposition} 

\begin{proof} We begin by establishing the following auxiliary statement.
\begin{lemma}\label{lem_conjugacy} Let $R$ be a ring satisfying $\Pic(R)=0$.
Let $\gH$ be a reductive $R$-group scheme.
Assume that $\gH$ admits a maximal torus $\gT$ which is split.
Then $\gH$ is split with respect to a Killing couple $(\gT, \gB)$.
Furthermore  split maximal tori (resp.\ Killing couples) of $\gH$
are $\gH(R)$--conjugate to $\gT$ (resp.\ $(\gT, \gB)$).
\end{lemma}

\begin{proof} 
In view of 
  \cite[XXII.5.5.1]{SGA3}, $\gH$ admits a Borel
  subgroup $\gB$ containing $\gT$. The fact that $\gH$ is split is  \cite[Ex.\ 5.1.4]{Co}. 
  For establishing the assertion regarding conjugacy, it is enough
  to show $\gH(R)$-conjugacy for Borel subgroups
  of $\gH$ and $\bB(R)$--conjugacy for maximal tori of $\gB$.
  The last fact is general \cite[XXVI.1.8]{SGA3}. As for the first, 
in view of  \cite[XXII.5.8.3]{SGA3}, the $R$-scheme $\gH/\gB$ represents
the $R$-functor of Borel subgroup schemes of $\gH$.
We have an exact sequence of pointed sets
$$
\gH(R) \to (\gH/\gB)(R) \to H^1(R,\gB)
$$
According to \cite[XXVI.2.3]{SGA3}, the map $H^1(R,\gT) \to H^1(R,\gB)$ is 
bijective; since $\gT \cong \GG_{m,R}^r$, we have 
$H^1(R,\gT)=\Pic(R)^r=0.$ The map $\gH(R) \to (\gH/\gB)(R)$ is thus onto and therefore
Thus $\gH(R)$ acts transitively on the set of Borel subgroups
of $\gH$.
\end{proof}

We now return to the proof of Proposition \ref{redGN}. The fact that $\tE$ admits reduction
$\hE$ such that  $_{\hE}\bT$ is isomorphic to $\gT'$ follows immediately from
\cite[lemme 2.6.2]{Gi4}. But we need a little bit more, namely that we can choose $\hE$ 
such that it is split by $Y/X$.

By our condition there is an isomorphism of group schemes $\phi_*:  \bG_Y \simlgr {_{\widetilde E}\bG}_Y$.
Then $\bT_Y$ and  $\phi_*^{-1}\bigl( \gT'_Y)$
are split maximal of $\bG_Y$. According  to Lemma \ref{lem_conjugacy},
$\bT_Y$ and 
$\phi_*^{-1}\bigl( \gT'_Y)$ are $\bG(Y)$--conjugated,
that is
$^g\bT_Y=\gT'_Y$ for some $g \in \bG(Y)$.
A fortiori they are $\tG(Y)$-conjugated.
 Inspection of the proof of \cite[lemme 2.6.2]{Gi4} 
shows that $g$ defines a reduction of $\tE$
to an $\bN$--torsor $\hE$ 
which is split by $Y$ 
and such that
$X$-torus $_{\hE}\bT$ is isomorphic to $\gT'$.
\end{proof}
\medskip





The quotient $\bN/\bT$ exists in the category of locally algebraic $k-$groups (\cite[$\rm{VI}_A$ 3.3.2 and 5.2]{SGA3}), and will be denoted  by $\bW.$  The kernel of the induced $k-$group morphism $\bN/\bT = \bW \to \bJ$ is $\bN_\bG(\bT)/\bT,$ a finite \'etale, in particular twisted constant, $k-$group that we will denote by $\bW_0$ \cite{SGA3}.

\begin{lemma}\label{WC} The $k-$group $\bW$ is twisted constant.
\end{lemma}
\begin{proof} We have seen that $\bW$ is locally algebraic. We have the exact sequence
$$1 \to \bW_0 \os{i}\longrightarrow \bW \os{h}\longrightarrow \bJ.$$
We can now conclude with the aid of Lemma \ref{Exactconstant}.
\end{proof}

\begin{remark} By \cite[Lemma 3.1]{GiPa} the map $h: \bW \to \bJ$ is surjective.
\end{remark}

 Recall that by Remark \ref{ks}(a), $X^{sc}$ comes equipped with $k_s-$scheme structure. 
 This yields a natural homomorphism $\bN(k_s) \to \bN(X^{sc})$ which is $\pi_1(X)-$equivariant since the action of $\pi_1(X)$ on these two groups is given by precomposition with the right action of $\pi_1(X)$ on $k_s$ and $X^{sc}$ respectively. For example, if $\alpha \in \pi_1(X)$ and $f \in \bN(X^{sc}) := \Hom_{Sch/X}(X^{sc}, \bN_X),$ then $({^\alpha}f )(x)= f (x^\alpha)$ for all $x \in X^{sc}.$ This yields a map of distinguished sets $\delta_\bN : H^1\big( \pi_1(X), \bN(k_s)\big) \to H^1\big( \pi_1(X), \bN(X^{sc})\big).$ 
 
 Similar considerations apply to $\bT$, and we have a map of pointed sets $$\delta_\bT : H^1\big( \pi_1(X), \bT(k_s)\big) \to H^1\big( \pi_1(X), \bT(X^{sc})\big).$$


The foregoing discussion and the exact sequence of $k-$groups
$$
1 \to \bT \to \bN \to \bW \to 1
$$
 yields the commutative diagram of pointed sets with exact rows

$$
\begin{CD}
H^1\bigl( \pi_1(X),  \bT(k_s) \bigr) @>{\rho}>>
H^1\bigl( \pi_1(X), \bN(k_s)\bigr) @>{\mu}>>
H^1\bigl( \pi_1(X), \bW(k_s) \bigr)  \\
@V{\delta_\bT}VV @V{\delta_\bN}VV  \mid \mid \\
H^1\bigl( \pi_1(X),\bT(X^{sc})\bigr)
@>{\rho'}>>  H^1\bigl( \pi_1(X), \bN(X^{sc})\bigr) @>{\mu'}>>
 H^1\bigl( \pi_1(X), \bW(X^{sc}) \bigr) \,. \\
\end{CD}
$$
 The last vertical arrow is the equality/identification  given by Corollary \ref{tor-loop}. 

By Proposition~\ref{redGN}, the torsor  $[\hE] $ is isotrivial. By Proposition \ref{discrete} we henceforth view $[\hE]$ as an element of  $H^1\bigl( \pi_1(X), \bN(X^{sc})\bigr).$ To establish  Theorem \ref{t->l} (ii),  it will thus suffice to show that $[\hE]$ is in the image of $\delta_\bN.$

Consider the image $[F'] \in H^1\big(\pi_1(X), \bW(X^{sc})\big)$ of $[\hE]$ under $\mu'$, and let $[F] \in H^1(\pi_1(X), \bW(k_s))$ be the element corresponding to $[F']$ under the bijection/identification $H^1\bigl( \pi_1(X), \bW(k_s) \bigr) \simeq H^1\bigl( \pi_1(X), \bW \bigr) = H^1\bigl( \pi_1(X), \bW(X^{sc}) \bigr).$ 

We divide the proof into two pieces:
\medskip

{\bf Step 1}:  $[F]$ has a preimage $[E] \in H^1(\pi_1(X), \bN(k_s)\bigr).$

{\bf Step 2}: If   $[F]$ has a preimage $[E] \in H^1(\pi_1(X), \bN(k_s)\bigr)$ the Theorem holds.

\medskip

{\bf Proof of Step 1}. By  \cite[Ch. I \S5.6 Prop. 4 ]{Se} the obstacle to the existence of $[E]$ is an element $[b] \in H^2(\pi_1(X), {_{F}}\bT(k_s) \big).$  The structure map $p_s : \Spec (A^{sc}) \to \Spec({k_s})$ we have constructed in Remark \ref{ks}(a) is a $\pi_1(X)-$module map, hence the natural group homomorphism $ {_{F}}\bT(k_s) \to {_{F}}\bT(A^{sc})$ induces a group homomorphism 
$ \eta : H^2(\pi_1(X), {_{F}}\bT(k_s) \big) \to H^2(\pi_1(X), {_{F}}\bT(A^{sc}) \big).$ Since $[F']$ has a lift, we see that $[b]$ belongs to the kernel of  $ \eta.$ It will thus suffice to show that $\eta$ is injective.

By Remark \ref{ks}(b) $p_s$  admits a section $\sigma$ that is $\pi_1(X)-$equivariant. This yields a group homomorphism 
$$ \rho : H^2(\pi_1(X), {_{F}}\bT(A^{sc}) \big) \to H^2(\pi_1(X), {_{F}}\bT(k_s) \big).$$
 
Since $\tau \circ \sigma = \Id_{\Spec(k_s)}$  we conclude that $\rho \circ \eta$ is the identity map of $H^2(\pi_1(X), {_{F}}\bT(k_s) \big).$ In particular, $\eta$ is injective.
\medskip

{\bf Proof of Step 2}. 
 Just like $[\hE]$, the image $[E'] := \delta_\bN([E]) \in H^1\big(\pi_1(X),\bN(X^{sc})\big)$ also maps to $[F']$ under $\mu'.$ To take advantage of this situation we will twist our relevant exact sequence of $H^1$s by $E$ and $E'$ as we now explain.
 


 \medskip

In addition to the ``constant" action of $\pi_1(X)$ on $\bN(X^{sc})$ described above, there is also a twisted action by $E'$ of $\pi_1(X)$ on $\bN(X^{sc})$ which, following Serre we denote by ${{^\alpha}'}f$ for $\alpha \in \pi_1(X)$ and $f \in \bN(X^{sc}).$ It is defined as follows: $E'$ is a cocycle in $Z^1\big(\pi_1(X), \bN(X^{sc})\big)$ and as usual we denote $E'(\alpha)$ by $E'_\alpha.$ Being a cocycle means that 
$E'_{\alpha \beta} = E'_\alpha {^\alpha}E'_\beta$ for all $\alpha, \beta \in \pi_1(X).$ The twisted action is then given by ${{^\alpha}'}f = E'_\alpha ({^\alpha}f) E'_{\alpha^{-1}}.$ The group $\bN(X^{sc})$ viewed with this twisted action will be denoted by ${_{E'}}\bN(X^{sc}). $ Thus ${_{E'}}\bN(X^{sc})$ {\it is the group} $\bN(X^{sc})$ {\it with a different} $\pi_1(X)-$action.\footnote{\,There is a twisted $X-$group ${_{E'}}\bN$ defined by the contracted product $E' \wedge^\bN \bN_X.$ Ignoring $\pi_1(X)-$actions, the $X^{sc}$ points of this $X-$group is the (abstract) group $\bN(X^{sc}).$ Our notation is thus unambiguous.} 
The twisted action is also continuous and one then defines $H^1\big(\pi_1(X), {_{E'}}\bN(X^{sc})\big)$ in the standard fashion. 

Since $\bT$ is a normal subgroup of $\bN$ conjugation by an $E'_\alpha$ leaves $\bT(X^{sc})$ stable. This allows us to define on the group $\bT(X^{sc})$ a twisted action of $\pi_1(X)$ that is denoted by ${_{E'}}\bT(X^{sc}),$ and corresponding cohomology $H^1\big(\pi_1(X), {_{E'}}\bT(X^{sc})\big).$ 
Similarly considerations apply to $\bW$ and $F'.$

\medskip

All the above considerations apply to $\bT,$  $\bN$ and $\bW$ if we replace $E'$ by $E,$ $F'$ by $F,$ and $X^{sc}$ by $k_s.$ 

We thus get  a twisted version of the above diagram, namely


$$
\begin{CD}\label{H2sequence}
H^1\bigl( \pi_1(X),  {_E}\bT(k_s) \bigr) @>{\rho}_E>>
H^1\bigl( \pi_1(X), {_E}\bN(k_s)\bigr) @>{\mu}_E>>
H^1\bigl( \pi_1(X), {_F}\bW(k_s) \bigr)  \\
@V\delta_\bT^EVV @V{\delta_\bN^E}VV \mid \mid  \\
 H^1\bigl( \pi_1(X),{_{E'}}\bT(X^{sc})\bigr)
@>{\rho_{E'}}>>  H^1\bigl( \pi_1(X), {_{E'}}\bN(X^{sc})\bigr) @>{\mu_{E'}}>>
 H^1\bigl( \pi_1(X), {_{F'}}\bW(X^{sc}) \bigr) \, .\\
\end{CD}
$$
\begin{remark}
 As maps of sets, the maps of this diagram {\it are the same} as before: $\rho_E = \rho, \, \delta_\bT^E  = \delta_\bT, ...$ . The subscript $E$ is used to indicate that they are being applied to classes of cocycles with respect to the twisted action of $\pi_1(X).$ 
 \end{remark}

There exists a twisting bijection 
$$\tau_{E'} :  H^1\big(\pi_1(X), \bN(X^{sc})\big) \simeq H^1\big(\pi_1(X), {_{E'}}\bN(X^{sc})\big)
$$ 
defined as follows: let $[Y] \in H^1\big(\pi_1(X), \bN(X^{sc})\big),$ where 
$Y\in Z^1\big(\pi_1(X), \bN(X^{sc})\big)$ is a cocycle. We define $\tau_E(Y) := Y' : \pi_1(X) \to \bN(X^{sc})$ by $\alpha \mapsto Y'_\alpha := Y_\alpha {(E'_\alpha)}^{-1}.$ One checks that this is a cocycle with respect to the twisted action, namely $Y'_{\alpha \beta} = Y'_\alpha {{^\alpha}'}{Y'_\beta}.$ That $\tau_{E'}$ is a bijection is easy to show. Note that $\tau_{E'} :[E'] \mapsto [1].$\footnote{\, The nature of the twisting bijection $\tau_{E'} : H^1(X, \bN) \to H^1(X, {_{E'}}\bN)$ is given in \cite[Ch. III, \S4, 3.4]{DG}) (where the inverse of our map is described is terms of contracted products). We are describing explicitly how this bijection looks like, in term of cocycles,  when restricted to isotrivial torsors.}
Similarly, we have a twisting bijection $\tau_E  :  H^1\big(\pi_1(X), \bN(k_s)\big) \simeq H^1\big(\pi_1(X), {_E}\bN(k_s)\big).$ 
\medskip

Since $\mu_{E'}(\tau_{E'} ([\hE])) = 1,$ there exists a class $[Y'] \in  H^1(\pi_1(X), {_{E'}}\bT(X^{sc})\big)$ such that $\rho_{E'}([Y']) = \tau_{E'}([\hE]).$ We claim that if there exists $[Y] \in H^1\bigl( \pi_1(X),  {_E}\bT(k_s) \bigr)$ such that $\delta_\bT^E([Y]) = [Y'],$ then  Theorem \ref{t->l}(ii)  holds, that is, $\hE$ is an $\bN-$loop torsor. Indeed, we need to show that, up to equivalence, $\hE_\alpha \in \bN(k_s)$ for all $\alpha \in \pi_1(X).$ Let $[Z] = \rho_E([Y]).$ Then $[\hE] = {\tau_{E'}}^{-1}\big( \tau_{E'}(\hE)) = {\tau_{E'}}^{-1}\big( \delta_\bN([Z])\big)$. Since $Z_\alpha E_\alpha \in \bN(k_s),$ the claim holds.

We are thus reduced to showing that $[Y]$ as in the claim exists. By construction $[Y'] \in H^1(\pi_1(X), {_{E'}}\bT(X^{sc})\big).$ Recall that by assumption the $X-$torus ${_{\hE}}\bT$ is split by a Galois extension $X'$ of $X$ of degree $m$ not divisible by $p.$ But then the same is true for the $X-$torus 
$_{E'}\bT.$ Indeed for all $\alpha \in \pi_1(X),$ $E'_\alpha = \hE_\alpha t_\alpha$ for some (unique) element $t_\alpha \in \bT(X^{sc}).$ Thus, conjugating $\bT(X^{sc})$ by $E'_\alpha$ or $ \hE_\alpha$ is the same.

\medskip 

\begin{lemma} $m[Y'] = 0.$ 
\end{lemma} 
{\it Proof} (of the Lemma).
Let $\tilde{\pi}$ be the kernel of the canonical group homomorphism $\pi_1(X) \to \Gal(X'/X).$\footnote{\, Since $X'$ is connected we can view $X^{sc}$ as the simply connected cover of $X'$. Then $\tilde{\pi} = \pi_1(X', a).$ See \cite[ IX Rem.5.8] {SGA3}.}  We have the exact sequence of groups
$$H^1(\pi_1(X), {_{E'}}\bT(X^{sc})\big) \os{\rm res} \longrightarrow H^1(\tilde{\pi}, {_{E'}}\bT(X^{sc})\big) \os{\rm cor} \longrightarrow H^1(\pi_1(X), {_{E'}}\bT(X^{sc})\big)$$
with ${\rm cor} \circ {\rm res}$ being equal to scalar multiplication by $m.$ The Lemma will follow if we can show that $[Y']$ is killed by the restriction map to $X'.$ To see that this holds observe that the $X-$torus ${_{E'}}\bT$ is split by $X',$ hence ${({_{E'}\bT)}}_{X'}$ is the product of $n-$copies of $\bG_{m, X'}.$ But then since by assumption $X'$ has trivial Picard group, we have ${H^1\big(\tilde{\pi}, {_{E'}}\bT(X^{sc})\big) \subset H^1\big(X', {({_{E'}}\bT)}}_{X'}\big) =0.$  \qed

Assume that $X = \Spec(A)$ and let $X^{sc} = \Spec(A^{sc}).$ We make some relevant observations:

(i) Since $X$ is normal the finite Galois extensions $X_i = \Spec(A_i)$ used to define $X^{sc},$ namely $A^{sc} = \limind A_i,$ are integral domains. Thus $A^{sc}$ is an integral domain.

(ii) The $m-$th power map $e : { A^{sc}}^{^\times}
\to { A^{sc}}^{^\times}$ is surjective. Indeed, if $a \in {A^{sc}}^\times$ then $B = A^{sc}[t]/<t^{m} - a>$ is a finite \'etale extension of $A^{sc}$, hence trivial (disjoint union of a finite number of copies of $X^{sc}$). Thus $a \in {({A^{sc}}^{\times})}^{m}.$ 

(iii) Let $ K \subset { A^{sc}}^{^\times}$ be the kernel of $e.$ Then $K = \bmu_{m}(k_s)$ since $A^{sc}$ is an integral domain. 
\medskip

Recall that the $\pi_1(X)-$module ${_{_{E'}}\bT(X^{sc}})$ is the group $\bT(X^{sc})$ with the twisted action of $\pi_1(X).$ Since $\bT \times_k X^{sc}$ is split,  ${_{E'}}\bT(X^{sc}) = { A^{sc}}^{^\times} \times \cdots \times{ A^{sc}}^{^\times};$ the product of $n-$copies of ${ A^{sc}}^{^\times}$ where $n = \dim_k(\bT).$ We will denote this Galois module by $N.$ 

Let us denote also by $e$ the $m-$power map of $\bT(X^{sc})$ (the $m-$th power in each component). Since the twisted action of $\pi_1(X)$ on  $\bT(X^{sc})$ is by group automorphisms, we get an exact sequence of $\pi_1(X)-$modules
$$
1 \longrightarrow M \longrightarrow N \os{e}\longrightarrow  N \longrightarrow 1.
$$
where $M$ is the $\pi_1(X)-$module ${_{_{E'}}}{_m}\bT(X^{sc}).$

Passing to cohomology yields
$$
H^1\big(\pi_1(X), M \big)  \longrightarrow H^1\big(\pi_1(X), N\big) \os{e^*}\longrightarrow H^1\big(\pi_1(X), N\big).
$$
By the last Lemma $[Y']$ is in the kernel of $e^*,$  hence it comes from an element $ [Y] \in H^1\big(\pi_1(X), M\big).$  But $M = {\bmu_{m}(k_s)}^n$ and the group ${\bmu_{m}(k_s)}^n$ with the twisted action of $\pi_1(X)$ is precisely the $m-$torsion submodule of ${_E}\bT(k_s).$ Thus $[Y] \in H^1\big(\pi_1(X), {_{E}}\bT(k_s)\big)$ and $\delta_\bT^E([Y]) = [Y']$ as desired. This completes the proof of Step 2.

\medskip

\noindent
{\bf Proof of the Theorem \ref{t->l} (iii)}.  The argument is analogous to the one given above if we  replace the universal cover
$X^{sc} \to X$ by its universal $p'$-geometric 
cover  $X^{sc,p'-geo} \to X$.

\begin{remark}\label{need} The assumption on the degree of the Galois extension $X'/X$ cannot be removed. We give two examples below. In the first example $\tG=\bG=\bT$ is a torus and in the second one $\tG=\bG={\bf PGL}_p$ is semisimple. Throughout $X = \Spec(k[t^{\pm 1}])$.
\medskip

\noindent
{\bf Example 1}.
Assume ${\rm char}(k) =p$ and $k$ admits a Galois extension $\ell$ of degree $p$ whose Galois group we denote by $\Gamma.$ We fix a generator $\sigma\in \Gamma$. Let $\bT_k=\cR_{\ell/k}(\bG_m)/\bG_m$.

\smallskip

\noindent
{\bf Claim 1}:  $\bT\simeq \cR^{(1)}_{\ell/k}(\bG_m)$. Indeed, the map $\cR_{\ell/k}(\bG_m)\to \cR^{(1)}_{\ell/k}(\bG_m)$ given by $x\to x^{1-\sigma}$ is surjective by Hilbert's Theorem 90 and its kernel is $\bG_m$.

\smallskip

\noindent
{\bf Claim 2}: One has $H^1(X,\bT)= k^\times/N_{\ell/k}(k^\times) \oplus \Z/p\Z.$ To prove the claim we first observe
that $\cR^{(1)}_{\ell/k}(\bG_m)$ (and hence $\bT$) is split over $X_\ell={\rm Spec}(\ell[t^{\pm 1}]),$ and that
${\rm Pic}(X_\ell)=1$. It follows
$$
H^1(X,\bT)=H^1(X,\cR^{(1)}_{\ell/k}(\bG_m))=H^1(\Gamma,\cR^{(1)}_{\ell/k}(\bG_m)(X_l))
$$
and hence
$$
\begin{array}{ccc}
H^1(X, \bT) &=&  
H^1(\Gamma , \ell[t^{\pm 1}]^\times \otimes_\Z X(\cR^{(1)}_{\ell/k}(\bG_m))_{*} )   \\ \nonumber
&= & H^1\big(\Gamma , (\ell^\times \oplus \Z) \otimes_\Z X(\cR^{(1)}_{\ell/k}(\bG_m))_{*}\big)  \\ \nonumber
&= & H^1\big(\Gamma , k^\times \otimes_Z X(\cR^{(1)}_{\ell/k}(\bG_m))_{*}) \oplus 
H^1\big(\Gamma,  X(\cR^{(1)}_{\ell/k}(\bG_m))_{*}\big)  \\ \nonumber
&&  \enskip = k^\times/N_{\ell/k}(k^\times) \oplus \Z/p\Z
\end{array}
$$

We let $[E]$ be a generator of the second summand $ \Z/p\Z.$

\smallskip

\noindent
{\bf Claim 3}: $[E]$ is not a loop torsor. Indeed,
the exact sequence
$$
1\longrightarrow \bG_m\longrightarrow \cR_{\ell/k}(\bG_m)\longrightarrow \bT\longrightarrow 1
$$
gives rise to
$$
1=H^1(X, \cR_{\ell/k}(\bG_m))\longrightarrow H^1(X,\bT)\stackrel{\delta}{\longrightarrow} H^2(X,{\bf G_m})
={\rm Br}(X).
$$ 
Tracing through the construction of the connecting homomorphism $\delta$ one can easily see
that  $\delta([E])$  is the class of the cyclic algebra
$A=(\ell/k,t)$. If $K=k(t)$ and $v$ is the discrete valuation of $K$ related to the variable $t,$
one can see that $A_{K_v}$ has nontrivial residue. In particular that $[A]\not=1$.

On the other hand, assume that $[E]$ is loop. Then it comes from an element of
$H^1(\pi_1(X),\bT(k_s))$. Passing from $k$ to $\tilde{k}=k^{-p^{\infty}}$ we can assume without loss
of generality that $k$ and $k_s$ are perfect. This implies that the $p$-power map $k_s\stackrel{p}{\longrightarrow}
k_s$ is bijective, hence that $H^2\big(\pi_1(X),\bT(k_s)\big)$ has no elements of order $p$. 

Consider next the commutative diagram
$$
\begin{CD}
H^1\bigl( \pi_1(X),  \cR_{\ell/k}(\bG_m)(k_s) \bigr) @>>>
H^1\bigl( \pi_1(X), \bT(k_s)\bigr) @>{\delta'}>>
H^2\bigl( \pi_1(X), \bG_m(k_s) \bigr)  \\
@VVV @V{\mu}VV  @V{\rho}VV \\
H^1\big( X,\cR_{\ell/k}(\bG_m)\big)
@>>>  H^1( X, \bT) @>{\delta}>>
 H^2(X, \bG_m ) \,. \\
\end{CD}
$$
Let $\alpha\in H^1\big(\pi_1(X),\bT(k_s)\big)$ be such that $\mu(\alpha)=[E]$. Since $[E]$ has
order $p$ and since $H^1\big(\pi_1(X),\bT(k_s)\big)$ is a periodic group we may assume that
$\alpha$ has order $p$. It follows that $\rho(\delta'(\alpha))=0$. On the other hand,
$$
\rho(\delta'(\alpha))=\delta(\mu(\alpha))=\delta([E])=[A]\not=1.
$$
This contradiction completes the proof of Claim 3.

\medskip

\noindent
{\bf Example 2.} We keep the same notation as in Example 1. Let $\bG={\bf PGL}_p$.
We have a canonical embedding $\bT\hookrightarrow \bG$. Consider the above constructed 
class $[E]\in H^1(X,\bT)$ and as usual we denote its image in $H^1(X,\bG)$ by 
$[E_{\bG}]$. As above we have a connected map $$\delta:H^1(X,\bG)\to H^2(X,\bG_m)={\rm Br}(X).$$
The image $\delta([E_{\bG}])$ of $[E_{\bG}]$ coincides with $[A]\not=1$. Arguing as above we conclude
that $E_{\bG}$ is not a loop torsor.

\end{remark}

\section{The toral nature of  loop torsors for extensions of locally constant $k-$groups by reductive groups}

Throughout this section $k$ is a field, and $p$  its characteristic exponent. We let $k'$ be a finite Galois extension of $k.$ Its Galois group will be denoted by $\Gamma.$

\subsection{Semilinear considerations.}

Let $\mathcal{A}$ be a $k'-$algebra (not necessarily associative, ...). A {\it semilinear automorphism of} $\mathcal{A}$
is an automorphism $\sigma$ of the underlying abelian group structure of $\mathcal{A}$ such that for all $a,b \in 
\mathcal{A}$ and $\lambda \in k'$
\medskip

SL1: $\sigma(ab) = \sigma(a)\sigma(b)$, and

SL2: there exists $\ts \in \Gamma$ such that $\sigma(\lambda a) =  {^{\ts}\lambda} \,\sigma(a).$
\medskip

If $\mathcal{A}$ is commutative associative and unital we also require that $\sigma(1) = 1.$ In this case $\mathcal{A}$ contains a canonical copy of $k'.$ In view of SL2 we then have
\medskip

SL2$'$ : $\sigma(\lambda) = {^{\ts}\lambda}$ for all $\lambda \in k'.$
\medskip

The element $\ts$ of SL2 is unique and is called the {\it type of} $\sigma.$ It is clear that the set 
$\Aut_\Gamma(\mathcal{A})$ of semilinear automorphisms of $\mathcal{A}$ is a group under composition, and that the type of a product is the product of the types.

 Let $H$ be an (abstract) group. A $\Gamma-${\it semilinear action  of} $H$ {\it on} $\mathcal{A}$ is a group homomorphism $\phi: H \to  \Aut_\Gamma(\mathcal{A}).$ If $h \in H$, we denote the type of $\phi(h)$ by $\tilde{h}$; it is an element of $\Gamma.$ The map $h \to \tilde{h}$ is a group homomorphism 
 $\tilde{\phi} : H \to \Gamma.$ 
  We denote the kernel of $\tilde{\phi}$ by $\tilde{H}.$ Thus $h \in \tilde{H}$ if and only if the action of $\phi(h)$ in $\mathcal{A}$ is $k'-$linear or, what is equivalent, if $\phi(h)$ is an automorphism of the $k'-$algebra $\mathcal{A}.$

The concept of semilinearity extends to coalgebras. We define it explicitly for the case that matters to us. Let $\bG$ an algebraic group over $k'.$ We denote for convenience the coordinate ring $k'[\bG]$ of $\bG$ by $A.$ Then $A$ is endowed with a Hopf algebra structure whose comultiplication we denote by $\Delta.$ 
 
 Let $\sigma$ be a semilinear automorphism of the associative commutative unital $k'-$algebra $A.$ It is simple to verify that there is a unique  additive map 
  $\sigma \otimes \sigma : A \otimes_{k'} A \to A \otimes_{k' }A$ 
 satisfying $\sigma \otimes \sigma(a \otimes b) =\sigma(a) \otimes \sigma(b).$ Moreover, $\sigma \otimes \sigma$ is a semilinear automorphism of type $\ts$ of the $k'-$algebra $A \otimes_{k' } A.$ We say that $\sigma$ is a {\it semilinear automorphism of $A$ (as a Hopf-algebra)} if
 \medskip
 
 HSL:  $\Delta \circ \sigma  = \sigma \otimes \sigma \circ  \Delta.$
 \medskip
 
  We denote the group of such automorphisms by $\Aut_\Gamma(A).$ Let $H$ be a group. A $\Gamma-${\it semilinear action} $H$ {\it   by automorphisms of} $\bG$ is a group homomorphism $H \to  \Aut_\Gamma(A).$
  \begin{remark} If $H$ acts $\Gamma-$semilinearly as automorphism of an algebra $\mathcal{A}$, then the action is 
  by $k-$linear maps, so it is an action by automorphisms of $\mathcal{A}$ viewed as a $k-$algebra. This observation 
  does not work for  semilinear actions on our Hopf algebras $A:$ There is in general no natural $k-$Hopf algebra structure on $A.$ Condition HSL is the right assumption to use.
  \end{remark}
  
 \begin{remark}\label{te} Let $\sigma \in \Aut_\Gamma(A),$ and let $\ts$ be the type of $\sigma.$ Recall that $\sigma$ acts on $k' \subset A$ like $\ts$ (see SL2$'$). We extend this action naturally to the ring of dual $k'-$numbers $k'[\epsilon] = k[\epsilon]\otimes_k k',$ $\epsilon^2 = 0.$ Note that the natural maps $k'[\epsilon] \to k'$ and $k' \to k'[\epsilon]$ commute with the action of $\sigma.$
 \end{remark}

  \begin{lemma}\label{SLD} Let $k'/k$, $\Gamma$ and $\bG$ be as above. Let $H$ be a group. A $\Gamma-$semilinear action of $H$ by automorphisms of $\bG$ induces in a natural way a $\Gamma-$semilinear action of $H$ by automorphisms of  the Lie algebra $\gg$ of $\bG$.
  \end{lemma} 
  \begin{proof}
  
  We begin by recalling, not just for the proof of the Lemma but for future reference,  the definition of  $\gg$ in terms of dual numbers. Let $\epsilon_{\bG} : A = k'[\bG] \to k'$ denote the counit map. As $k'$-spaces we have $A = k' \oplus I_{\bG}$ where $I_{\bG}$ is the kernel of $\epsilon_{\bG}.$  
  
  An {\it $\epsilon_\bG-$derivation of $A$} is an element of $\Der_{k'}(A, k')$ where $k'$ is viewed as a $A$--module via the counit map $\epsilon_{\bG}.$ In other words, an $\epsilon_\bG-$derivation of $A$ is a $k'-$linear map $\delta : A \to k'$ satisfying 
  $\delta(ab) = \epsilon_{\bG}(a)\delta(b) + \epsilon_{\bG}(b)\delta(a)$
  for all $a,b \in A.$ 
    
  Consider the ring of dual $k'-$numbers $k'[\epsilon].$ The ring homomorphism $k'[\epsilon] \to k'$ given by $\epsilon \mapsto 0$ induces a group epimorphism $\bG(k'[\epsilon]) \to \bG(k').$ The kernel of this morphism is by definition $\gg.$ It consists of all elements of $\bG(k'[\epsilon])$ of the form
$$
a \mapsto \epsilon_{\bG}(a) + \delta(a)\epsilon 
$$
where $a \in A$ and $\delta$ is an  $\epsilon_\bG-$derivation of $A.$  
We denote by $\delta_x$ the $\epsilon_\bG-$derivation corresponding to  $x \in \gg,$ that is $x = \epsilon_{\bG} + \delta_x\epsilon.$ The $k'-$space structure is such that $\lambda x$ corresponds to $\lambda\delta_x.$ 

The bracket of $\epsilon_\bG-$derivations is  defined by
$$
[\delta_1, \delta_2] = \mu_A \circ (\delta_1 \otimes \delta_2 - \delta_2 \otimes \delta_1) \circ \Delta.
$$
where $\mu_A$ is the multiplication of the $k'-$algebra $A.$
In terms of groups, if $x_i = \epsilon_{\bG} + \delta_i\epsilon_i$, $i = 1,2$, then in $\bG(k'[\epsilon_1, \epsilon_2])$ we have
$$ x_1x_2x_1^{-1}x_2^{-1} =  \epsilon_{\bG} + [\delta_1, \delta_2]\epsilon_1\epsilon_2.$$

We now turn to the proof of the Lemma.  If $\delta$ is an $\epsilon_\bG-$derivation and $\sigma \in \Aut_\Gamma(A),$ it is straightforward to verify that ${^{\sigma}\delta} := \ts \circ \delta \circ \sigma^{-1}$ is an 
$\epsilon_\bG-$derivation. If $x \in \gg$ we let ${^{\sigma}x}$ be the element of $\gg$ that corresponds to ${^{\sigma}(\delta_x)}.$ Thus ${^{\sigma}x} =  \epsilon_\bG + {^{\sigma}(\delta_x)}.$

It is clear  that ${^{\sigma}(\delta_1 + \delta_2)} = {^{\sigma}\delta_1} + {^{\sigma}\delta_2},$ and that ${^{\sigma}[\delta_1, \delta_2]} = [{^{\sigma}\delta_1}, {^{\sigma}\delta_2}].$ Furthermore, if $\lambda \in k',$ $x \in \gg,$ we have 
  $$\big({^{\sigma}(\lambda \delta})\big)(a) = \ts\big(\lambda \delta(\sigma^{-1}(a))\big) = {^{\ts}\lambda}\, {^{\sigma}\delta}(a).$$ This translates, under our identification of $\gg$ with $\epsilon_\bG-$derivations of $A,$ into a semilinear action of $\Gamma$ by automorphism of $\gg.$
  
   We have thus shown that the elements of  $\Aut_\Gamma(A)$ can be naturally seen as elements of $\Aut_\Gamma(\gg).$ Since this process is visibly a group homomorphism, the Lemma follows. \end{proof}
   
   \begin{remark} A slightly different proof of this Lemma is given in \cite[Cor. 4.7]{GP0}. The above argument is more direct and sufficient for our purposes. It also makes some of the notation and argument in subsequent proofs easy to follow.
      \end{remark}

  Let $\phi$ be a $\Gamma-$semilinear action of $H$ (on an algebra or an algebraic group, as described above). We say that $\phi$ (or $H$ by abuse of language)  is {\it  super solvable}, or simply {\it of type} $\rm {(S)},$ if there exists a finite family of subgroups $H_s \supset H_{s - 1} \supset ... \supset H_1
\supset H_0 = 0$ of $H$  satisfying the following condition:
\medskip

{\rm (S)}   $\ker(\tilde{\phi}) = H_s = \tilde{H}.$ Furthermore, each $H_i$ is normal in $H$ and  the quotients $H_i /
H_{i-1}$ are cyclic.
\medskip

\begin{remark}\label{onS} Recall that $\tilde{H}$ is the subgroup of $H$ consisting of elements of trivial type. The condition $H_s = \tilde{H}$ then says that, in the case of algebras, the elements of $H_s$ act as $k'-$linear transformations. In other words, they act as automorphisms of the $k'-$algebra under consideration.
\end{remark}

\begin{remark}\label{sigmaT} Let $\sigma \in \Aut_\Gamma(A).$ Let $\bT$ be a $k'-$subgroup of $\bG.$ Then $k'[\bT] = A/I$ for some (unique) Hopf ideal $I$ of $A.$ View $\sigma$ as an automorphism of $\Spec(k').$ By base change we get a new $k'-$group ${{^\sigma}\bT}.$ 

By definition $k'[{^{\sigma}\bT}] = 
k'[\bT] \otimes_{k'} k' = (A/I) \otimes_{k'} k'$
where $k'$ is viewed as a ring extension of $k'$ via the ring (iso)morphism $\sigma : k' \to k';$ that is,  $(\lambda a + I) \otimes 1 = (a + I) \otimes {^{\tilde{\sigma}}\lambda}$ for all $a \in A$ and $\lambda \in k'.$ 

One easily verifies that $\sigma(I)$ is a Hopf ideal of $A,$ and that the map
$(A/I) \times k' \to A/\sigma(I)$ given by $(a +I, \lambda) \mapsto {^{\tilde{\sigma}}\lambda}\, \sigma(a) + \sigma(I)$ induces a $k'-$Hopf algebra isomorphism $k'[{^{\sigma}\bT}] \simeq A/\sigma(I).$ It follows that we can canonically identify ${^{\sigma}\bT}$ with a $k'-$subgroup of $\bG.$ Furthermore, $\sigma$ is a $k'-$group isomorphisms between $\bT$ and ${^\sigma}\bT.$ In particular,  if $\bT$ is a maximal torus of $\bG$ so is ${{^\sigma}\bT}.$

\end{remark}
  
  \subsection{Semilinear version of a theorem of Borel-Mostow and Winter}

If $\gg$ is a Lie algebra over a field extension $K$ of $k$ and $d \in \Der_K(\gg)$, then $d^p$ (viewed as a $K$-linear endomorphism of $\gg$) is, by Leibnitz rule, actually an element of $\Der_K(\gg).$
Following \cite{W} we say that $\gg$ is {\it w-restricted} if the $p$-th power of all inner derivations is inner:  if $x \in \gg,$ there exists $y \in \gg$ such that $(\ad_\gg \,x)^p = \ad_\gg(y).$\footnote{\, The terminology stands for {\it weakly restricted}. The concept of restricted Lie algebra over a field of positive characteristic is due to Jacobson (see \cite{Jac}). Restricted Lie algebras are $w-$restricted.}

\begin{theorem}\label{BMW} (Semilinear Borel-Mostow-Winter)

 Let $\gg$ be a finite dimensional $w-$restricted $k'$-Lie algebra. Let $\phi$ be a $\Gamma-$semilinear action of a group $H$ by automorphisms of $\gg.$ Assume that:
 \medskip
 
 {\rm (a)} $H$ is of type {\rm (S)}. 
 
 {\rm (b)} The elements of $H_s$ act semisimply on $\gg.$\footnote{\, Because $H_i \subset
H_s = \tilde{H}$ (see Remark \ref{onS}), the action of the elements of $H_i$  on $\gg$ is
$k'$--linear. The assumption is that this action is semisimple
as  a $k'$--linear endomorphisms of $\gg.$}
\medskip

\noindent Then there  exists a Cartan subalgebra  of $\gg$ which is stable under the action of $H$.
\end{theorem}

\begin{proof} Observe for future reference  that if $\mathfrak{f}$ is a Lie subalgebra of $\gg$ that is stable under the action of $H$, then we have a natural induced semilinear action of $H$ by automorphisms of $\mathfrak{f}$ with the desired properties, namely (a) and (b) above.

 If $s = 0$ we can by mean of condition $(\rm{S})$ identify $H$
with a subgroup $\Gamma_0$ of $\Gamma$ via the type: $\Gamma_0 = \tilde{\phi}(H).$ Thus, if $h \in H$ and  $\phi(h) = \gamma,$  then 
\begin{equation}\label{SL2'}
 {^{h}}(\lambda x) := {^{\gamma}}(\lambda x) = {^{\gamma}\lambda}{^h}x
 \end{equation}
for all $x \in \gg, \lambda \in k'.$

Let $k_0 = k'^{\Gamma_0}.$ This yields a
$k'/k_0-$semilinear action of $\Gamma_0$ by automorphisms of $\gg.$ By Galois descent the
fixed point $\gg_0 := \gg^{\Gamma_0}$ is a Lie algebra over $k_0$ for which
the canonical map $\rho : \gg_0\otimes_{k_0} k'\simeq \gg$
is a $k'$--Lie algebra isomorphism. It is easy to see by Galois descent that $\gg_0$ is $w-$restricted (see also \cite{Jac} pg. 192). By  [W] \S6 
the $k_0-$Lie algebra $\gg_0$ admits a Cartan subalgebra $\gh_0$. Then $\rho(\gh_0 \otimes_{k_0} k')$ is a Cartan
subalgebra of $\gg$ which is $H$--stable as one can easily see from (\ref{SL2'}).

Assume henceforth that $s \geq 1.$ We will reason by induction on $ \ell = \dim(\gg).$ The result holds if $\ell = 0.$ Assume $\ell \geq 1$ and that the result holds for all $w-$restricted Lie algebras of dimension less than $\ell$ equipped with a semilinear action of $H$ satisfying conditions (a) and (b) of the Theorem.

If $\gn$ is a nilpotent subalgebra of $\gg$ we will denote by $\gg^0(\gn)$ the null-space of $\gn$ acting on $\gg$ via the adjoint representation, that is
$$\gg^0(\gn) = \{ y \in \gg :  \forall x \in \gn, \,\,(\ad \, x)^n (y) = 0 \,\, \text {\rm for $n$ large enough}\}.$$

Consider a generator $\theta$ of the cyclic group
$H_1.$ We may assume that $\theta \neq 1$ (we can delete all the $H_i$ for which $H_i = H_{i -1}$ and produce a new SLD with smaller $s$).  As we have already
observed, the action of $\theta$ on $\gg$ is $k'$--linear.   The fixed point $k'-$ Lie subalgebra $\gg_1 = \gg^\theta$ of $\gg$ is $w-$restricted [W] \S6, second Lemma following Cor. 4. Let $\gh_1$ be a Cartan subalgebra of $\gg_1$  stable under $H.$ Then $\gg_2 = \gg^0(\gh_1)$ is $w-$restricted \cite{W} 
Theorem 5. Moreover, $\gg_2$ is a solvable Lie algebra by \cite{W} Theorem 2. It is clear that $\gg_2$ is $H-$stable.

Case 1: $\gg_2 \neq \gg.$ By induction there is an $H-$ stable Cartan subalgebra $\gh_2$ of $\gg_2.$ Then $\gh := \gg^0(\gh_2)$ is $H-$ stable, and it is a Cartan subalgebra of $\gg$ by \cite{W} Theorem 5.

\medskip

Case 2: $\gg_2 = \gg.$ Let $\ga$ be the last non-zero term of the derived series of $\gg.$ It is an abelian ideal that is $H-$stable. Let $\overline{\gg} = \gg/\ga.$ Being a quotient of a $w-$restricted Lie algebra, $\overline{\gg}$  is $w-$restricted. Since $\ga$ is $H-$stable there is a naturally induced semilinear action of $H$ on $\overline{\gg}.$ Moreover, the canonical surjective Lie algebra morphism $^{-} :\gg \to \overline{\gg}$ is $H-$equivariant.  Since $\dim(\overline{\gg}) < \dim(\gg)$ there exists an $H-$stable Cartan subalgebra $\overline{\gb}$ of $\overline{\gg}.$ Let $\gb$ be the preimage of $\overline{\gb}$ in $\gg.$ It is $H-$stable, hence has an induced semilinear action.
\medskip

\noindent Claim:  {\it $\gb$ is $w-$restricted.}

\noindent Proof (of the claim). Let $x \in \gb$ and chose $y \in \gg$ such that $(\ad_\gg \,x)^p = \ad_\gg \,y.$ Then $[\overline{y} \, ,\, \overline{\gb}] \subset \overline{\gb}.$ Since $\gb$ is a Cartan subalgebra of $\overline{\gg}$ it follows that $\overline{y} \in \overline{\gb}$, hence that $y \in \gb.$ Thus $(\ad_\gb \,x)^p = \ad_\gb \,y.$
\medskip

Case 2(i): $\dim(\gb) < \dim(\gg). $ Let $\gh$ be an $H-$stable Cartan subalgebra $\gb.$ We claim that $\gh$ is a Cartan subalgebra of $\gg$.  By \cite{Bbk} Ch. 7 \S2 Cor. 4, $\overline{\gh}$ is a Cartan subalgebra of $\overline{\gb}.$ But $\overline{\gb}$ is its own Cartan subalgebra, so $\overline{\gh} = \overline{\gb}.$ We need to show that $\gh$ is self-normalized in $\gg.$ Assume $x \in \gg$ is such that $[x, \gh] \subset \gh.$ Then $[\overline{x}, \overline{\gh}] \subset \overline{\gh}.$ Since $\overline{\gh} = \overline{\gb}$ and $ \overline{\gb}$ is a Cartan subalgebra of  $\overline{\gg}$, we get that $\overline{x} \in \overline{\gb}.$ Then $x \in \gb$ and therefore $x \in \gh$ since $\gh$ is a Cartan subalgebra of $\gb.$

Case 2(ii): $\gb = \gg.$ This can only happen if $\overline{\gb}  = \overline{\gg},$ so $\overline{\gg}$ is nilpotent. Let $\gh$ be a Cartan subalgebra of $\gg$. As we saw $\overline{\gh}$ is a Cartan subalgebra of $\overline{\gg}$, hence equal to $\overline{\gg}$ since $\gg$ is nilpotent. It follows that $\gg = \gh  + \ga.$

View $\ga$ as an $\gh-$module. By Engel's Theorem there exists a non-zero element  $z \in \ga$ such $[x,z] = 0$ for all $x \in \gh.$ Since $\ga$ is abelian we see that the centre $\gz$ of $\gg$ is not trivial. Note that since $\gz$ is $H-$stable, $\gg/\gz$ and has an induced semilinear action of $H.$ By induction there exists an $H-$stable Cartan subalgebra $\overline{\gh}$ of $\gg/\gz.$ The preimage $\gh$ of $\overline{\gh}$ in $\gg$ is $H-$stable. It is also a Cartan subalgebra of $\gg$ by the reasoning of Case 2(i). 
  \end{proof}

\begin{remark} If $H = H_s$  the Theorem reduces
to the ``Main result (B)" of Borel and
Mostow \cite{BM} for $\gg$  (in characteristic $0$), and Theorem 6 of \cite{W} (in arbitrary characteristic).
\end{remark}
\subsection{An application to reductive $k-$groups}
We want to use the last Theorem to prove the existence of $H-$stable tori on reductive $k'-$groups. We maintain the previous notation. Thus, $\bG$ is an algebraic group over $k'.$ As before we denote its coordinate ring $k'[\bG]$ by $A,$ and the comultiplication of the Hopf algebra structure of $A$ by $\Delta.$

Let $\phi : H \to \Aut_\Gamma(A)$ be a $\Gamma-$semilinear action of $H$ on $\bG.$ If $\bT$ is a $k'-$subgroup of $\bG$ we say that $\bT$ is $H-${\it invariant}, or {\it stable under} $H$, if ${{^\sigma}\bT} = \bT$ for all $\sigma \in \phi(H).$
\begin{lemma}\label{points} Let $\bT$ be a maximal $k'-$torus of $\bG.$ Let $\sigma \in \Aut_\Gamma(A)$ and consider the maximal torus ${^\sigma{\bT}}$ of $\bG$ defined in Remark \ref{sigmaT}.
The map $x \mapsto {^\sigma{x}}:= \ts \circ x \circ \sigma^{-1}$  is a  group isomorphism between $\bT(k')$ and ${{^\sigma}\bT}(k').$ The analogous map applied to $\epsilon_\bT-$derivations yields an isomorphism between the Lie algebra of $\bT$ and that of ${{^\sigma}\bT}.$ 
\end{lemma}
\begin{proof} We have seen in Remark \ref{sigmaT} that $\sigma$ is a $k'-$group  isomorphism between $\bT$ and  ${^\sigma}\bT.$ We look at the explicit nature of this isomorphism at the level of $k'-$points.
Let $x \in \bT(k') = \Hom(A, k').$ Note that $ \sigma^{-1}$ induces  a $k-$algebra isomorphism, also denoted by $\sigma^{-1},$ between $A/\sigma(I)$ and $A/I.$ Thus $f := x \circ \sigma^{-1} : A/\sigma(I) \to k'$ is a ring homomorphism satisfying 
 $f \big((\lambda {a + \sigma(I)\big) = {^{\ts^{-1}}\lambda}}f\big(a +\sigma(I)\big).$  Since the action of $\sigma$ on $k' \subset A$ is given by $\ts$ (see SL2$'$), we see that ${^\sigma{x}} \in  \Hom(A/\sigma(I), k') =  {{^\sigma}\bT}(k').$

Assume next that $\delta : A/I  \to k'$ is an $\epsilon_\bT-$derivation. One easily checks that ${{^\sigma}\delta} := \ts \circ \delta \circ \sigma^{-1} : A/\sigma(I) \to k'$ is  an $\epsilon_{{{^\sigma}\bT}}-$derivation (see Remark \ref{te}). The map $\epsilon_\bT + \delta \epsilon \mapsto \ts \circ (\epsilon_\bT + \delta  \epsilon) \circ \sigma^{-1} = \epsilon_\bT + {{^\sigma}\delta}\epsilon$ is the desired Lie algebra isomorphism. 
\end{proof}
\begin{proposition}\label{stabletori} Let $\bG$ be a reductive $k'-$group and $A = k'[\bG]$ its ring of regular functions. Let $\phi : H \to \Aut_\Gamma(A)$ be a $\Gamma-$semilinear action of $H$ on $\bG.$ Assume that $H$ is of type \rm{(\rm S)} and that under the induced action of $H$ on the Lie algebra $\gg$ of $\bG$ the elements of the kernel of $\phi$ act semisimply.  Then there exists a maximal torus $\bT$ of $\bG$ that is $H-$stable.
\end{proposition}
\begin{proof} We denote the Lie algebra of $\bG$ by $\gg.$
We claim that $\gg$ is a $w-$restricted Lie algebra over $k'.$ 
  By definition, $w-$restricted  means that  the derivation $(\ad_\gg \,x)^p$ of $\gg$ is inner for all $x \in \gg.$  Since $\gg$ is a restricted Lie algebra (\cite{DG} II \S7.3.4), $\gg$ has a $p$-map $x \mapsto x^{[p]}$ which satisfies $(\ad_\gg \,x)^p = \ad_\gg \,x^{[p]}.$
 
 By Theorem \ref{BMW} there exists a Cartan subalgebra $\gt$ of $\gg$ that is $H-$stable. 
 
Assume first that $\bG$ is semisimple of adjoint type. By \cite{SGA3} XIV Theo. 3.9 and 3.18 there exists a unique maximal torus $\bT$ of $\bG$ whose Lie algebra is $\gt.$ Because of the uniquenes, to show that $\bT$ is $H-$invariant it will suffice to show that for all $\sigma \in  \Aut_\Gamma(A)$ the Lie algebra $\gt_\sigma$ of ${^\sigma{\bT}}$ coincides with $\gt.$ But this is clear from Lemma \ref{points}. Indeed, if $y \in \gt_\sigma \subset \gg,$ then $ y = {^\sigma{x}}$ for some $x \in \gt.$  We have $[y, \gt] = [{^\sigma{x}}, \gt] = {^\sigma}[x,  {^{\sigma^{-1}}\gt}] = {^\sigma}[x, \gt] \subset  {^\sigma{\gt}}= \gt.$ Since since $\gt \subset \gg$ is self-normalized we conclude that $y \in \gt.$

 Assume next that $\bG$ is arbitrary. The centre $\bZ$ of $\bG$ is $H-$stable and we get an induced $H-$semilinear action of $H$ on $\bG/\bZ. = \bG_0.$ Since $\bG_0$ is of adjoint type there exists a maximal torus $\bT_0$ of $\bG_0$ that is $H-$stable. Let $\bT$ be the unique maximal torus of $\bG$ that maps to $\bT_0$ under the quotient map $q : \bG \to \bG_0.$ Since $q$ is $H-$equivariant we conclude that $\bT$ is $H-$stable.
\end{proof}

\subsection{An application to extensions of constant $k-$groups by reductive $k-$groups} \label{seckloopgroup}

As we have pointed out already, if $k$ is a field of characteristic zero the ring $R = k[t_1^{\pm 1}, \cdots, t_n^{\pm 1}]$ satisfies the following property: every finite \'etale covering $Y\to {\rm Spec}(R)$ has trivial Picard group.  This fact combined with a (characteristic 0) Proposition \ref{stabletori} is used in \cite{GP0} to show that loop reductive $R-$groups admit maximal tori. We will now establish the analogue of that result in our present setting for more general base schemes and groups.
\medskip

We maintain all the notation of \S 4. In particular we have our exact sequence (\ref{mainsequence}),
$$
1 \to \bG \to \tG \to \bJ \to 1.
$$

Let $X = \Spec(R)$ be an affine geometrically connected $k-$scheme admitting a rational point.  Let $[\tE] \in H^1(R,\tG)$ be a loop torsor, and consider the twisted $R-$group $_{\tE}\tG.$ Since loop torsors are isotrivial there exists a (finite) Galois extensions $R'$ of $R$  that trivializes $_{\tE}\tG,$ that is $_{\tE}\tG \times_R R' \simeq \tG \times_k R'.$ 
Since $\bG$ is a normal subgroup of $\tG$ we have a natural group homomorphisms $\tG(k_s) \to \bAut(\bG)(k_s)$ which is visibly $\pi_1(R)-$equivariant. This yields a map of pointed sets
\begin{equation}\label{map}
H^1\big(\pi_1(R), \tG(k_s)\big) \to H^1\big(\pi_1(R), \bAut(\bG)(k_s)\big).
\end{equation}
We denote by $[E] \in H^1\big(\pi_1(R), \bAut(\bG)(k_s)\big)$ the image of $[\tE]$ under this map, and consider the reductive $R-$group ${_E}\bG.$ We observe for future reference that $R'$ also trivializes the $R-$group ${_E}\bG.$


Let $u \in Z^1\big(\pi_1(R), \bAut(\bG)(k_s)\big)$ be a cocycle representing $[E].$ We know that $u$ factors through the Galois extension $R'$ of $R,$ whose Galois group we will denote by $\Gamma'.$ The image of $\Gamma'$ on $\Gal(k)$ under the map $p : \pi_1(R) \to \Gal(k)$ of (\ref{fundamentalexact}) corresponds to a finite Galois extension $k'$ of $k$ whose Galois group we denote by $\Gamma.$ 

Let us denote by $\bG'$ the $k'-$group $\bG_{k'} = \bG \times_k k'.$ The cocycle $u$ takes values in $\bAut(\bG)(k')$ and we will henceforth (anti)identify   $\bAut(\bG)(k')$ with the automorphisms of the $k'-$Hopf algebra $k'[\bG']$ of $\bG'.$   

We denote the kernel of the induced map $\overline{p} : \Gamma' \to \Gamma$ by $\overline{\Gamma}.$ For convenience if $\gamma' \in \Gamma'$ we denote $\overline{p}(\gamma')$ simply by $\gamma.$ This is the type of $\gamma'.$ The group $\Gamma'$ will play the role of the group $H$ of Theorem \ref{BMW} as we now explain.

Recall that $\bG' = \bG \times_k k'$ and that $k'[\bG'] = k[\bG] \otimes_k k'.$ Consider for each $\gamma' \in \Gamma'$ the map $\phi(\gamma') : k'[\bG'] \to k'[\bG']$ defined by
$$
\phi(\gamma') =  u_{\gamma'} \circ (\id \otimes\gamma).
$$
Since each $u_{\gamma'}$ is an automorphism of the $k'$--Hopf algebra $k'[\bG'],$ it follows that via $\phi$ we obtain an action of $\Gamma'$ on $k'[\bG']$ by $\Gamma-$semilinear Hopf algebra automorphisms. We thus get an induced $\Gamma-$semilinear action of $\Gamma'$ on the $k'$-Lie algebra $\gg'$ of $\bG'.$

\begin{theorem}\label{existenceoftori} Let $X = \Spec(R)$ be an affine  geometrically connected noetherian $k-$scheme such that $X(k) \neq \emptyset.$ Let $[\tE] \in H^1(X, \tG)$ be a loop torsor and $\Gamma'$ the Galois group of  a Galois extension $R'$ of $R$ that trivializes $[\tE],[E]$ and $u$. Assume that:

(i) The group $\Gamma'$ is of type {\rm (S)}.

(ii) The elements of $\overline{\Gamma}$ (the kernel of the type map) act on $\gg'$ as semisimple $k'$-linear transformations.\footnote{\,  The elements of $\overline{\Gamma}$ act as elements of finite order, hence semisimply in  characteristic $0$.}

\noindent Then the $R-$group  $_{\tE}\tG$ admits a maximal torus that splits over $R'\otimes_k k_s$.
\end{theorem}

\begin{proof}  

The $R-$group ${_{\tE}}\tG$ is a twisted from of $\tG$, hence locally of finite presentation. The concept of maximal tori of $\tG$ is defined in \S\ref{tt}. The $R-$group ${_E}\bG$ defined via (\ref{map}) is precisely the group $_{\tE}\bG$.
It follows that to establish the theorem it will suffice to show that the reductive $R-$group ${_E}\bG$ admits a maximal torus.

Consider the reductive $R'$-group $\gG =  \bG \times_k R'.$ The coordinate ring of the $R-$group ${_{E}}\bG$ is obtained from that of $R'[\gG]$ by Galois descent, namely
$$
R[{_{E}}\bG] = \{x \in R'[\gG] : u_{\gamma'} \, ^{\gamma'}x = x  \, \,\forall\, \gamma' \in \Gamma' \}.
$$
Note that $R'[\gG] = k[\bG] \otimes_k R'$ and that, under this identification,  the action of $\Gamma'$ on $R'[\gG]$ is given by 
\begin{equation}\label{R[G]s}
^{\gamma'}(a \otimes s) =  u_{\gamma'}(a \otimes 1) (1 \otimes {^{\gamma'}}s)
\end{equation}
for all $ \gamma' \in \Gamma',$  $a \in k[\bG]$ and $ s \in R'.$ The right-hand-side of (\ref{R[G]s})  is the product of two elements  of $k[\bG] \otimes_k R',$ and the equality holds because, since $u_{\gamma'}$ is $R'-$linear, 
$u_{\gamma'}(a \otimes {^{\gamma'}}s) =  u_{\gamma'}\big((a \otimes 1) \otimes (1 \otimes {^{\gamma'}}s)\big) =  u_{\gamma'}(a \otimes 1)(1 \otimes {^{\gamma'}}s).$ 

 As we shall see, it is necessary for us to understand this action  in terms of the coordinate ring $k'[\bG'] = k[\bG] \otimes_k \, k'$ of the $k'-$group $\bG',$ that is, under the identification
\begin{equation}\label{?} 
k[\bG] \otimes_k R' \simeq  k[\bG] \otimes_k k' \otimes_{k'} R' =  k'[\bG'] \otimes_{k'} R'.
\end{equation} 
\begin{lemma}\label{??} Under the identification {\rm (\ref{?})} the action {\rm (\ref{R[G]s})} is given by
$$^{\gamma'}(a' \otimes s) =  \phi(\gamma')(a')  \otimes {^{\gamma'}}s$$
for all $a' \in k'[\bG']$ and $s \in R'.$
\end{lemma}
\noindent{\it Proof} (of the Lemma). Since $k'[\bG'] = k[\bG] \otimes_k \, k'$ the element $a'$ is a sum of elements of the form $a \otimes z $  with $a \in k[\bG]$ and $z \in k'.$ 
By definition $\phi(\gamma')(x \otimes z) = u_{\gamma'}(a \otimes {^{\gamma}}z) = u_{\gamma'}(a \otimes 1) (1 \otimes {^{\gamma}}z).$ On the other hand if we view $z$ as an element of $R',$ then ${^{\gamma'}}z = {^{\gamma}}z$ so that ${^{\gamma}}z  {^{\gamma'}}s = {^{\gamma'}}z  {^{\gamma'}}s = {^{\gamma'}}(zs).$ The Lemma follows.\qed
\medskip

By Proposition \ref{stabletori} there exists a maximal torus $\bT'$ of $\bG'$ that is stable under the action of $\Gamma'.$ The coordinate ring $k'[\bT']$ of $\bT'$ is of the form $k'[\bG']/I$ for some Hopf ideal $I$ of $k'[\bG']$ that is stable under the action of $\Gamma'$ (see Remark \ref{sigmaT}). It follows that the reductive $R'$-group $\bG' \times_{k'} R'$ admits the maximal torus   $\gT = \bT' \times_{k'} R'.$ But since  
$\bG' \times_{k'} R' = \bG \times_k k' \times_{k'} R' = \bG  \times_k R' = \gG,$ we see that $\gT$ is a maximal torus of $\gG.$  We show that ${_{E}}\bG$ has a maximal torus by showing that $\gT$ descends, namely it is stable under the action of $\Gamma'.$ More precisely, $\gT$ corresponds to the ideal $I \otimes_{k'} R'$ of the $R'-$Hopf algebra $k'[\bG'] \otimes_{k'} R',$ and we need to show that this ideal is stable under the action of $\Gamma'.$ That this is the case follows from Lemma \ref{??} since $I$ is stable under the semilinear action of $\Gamma'.$ It remains to note
that since $\bT'$ is split over $k_s$ the constructed maximal torus in $_E\bG$ is split over 
$R'\otimes_k k_s$. This completes the proof of the theorem.
\end{proof}
\begin{remark}\label{semisimplicity} If $\overline{\Gamma}$ is a finite group of order prime to $p$
then its elements automatically act  on $\gg'$ as semisimple $k'$-linear transformations.
\end{remark}


\section{The Laurent polynomials case}

In this section we consider the ring $R_n=k[t_1^{\pm 1}, \dots, t_n^{\pm 1}]$ for $n \geq 1$.  
We remind the reader that the $p'-$geometric simply connected cover  $R_n^{p'-geo}$ of $R_n$
is the inductive limit of the $l[t_1^{\pm 1/m}, \dots, t_n^{\pm 1/m}]$
for $l/k$ running over the finite Galois 
subextensions of $k_s$ and
$m$ coprime to $p$ such that $\mu_m(l)= \mu_m(k_s)$\cite[Remark 4.2]{G}. 
Such an extension is Galois and 
$\Gal( l[t_1^{\pm 1/m}, \dots, t_n^{\pm 1/m}]/R_n)=
\mu_m(l)^n \rtimes \Gal(l/k)$. 

\begin{theorem}\label{thm_laurent} Let $1 \to \bG \to \widetilde \bG \to \bJ \to 1$ be an exact 
sequence of $k$-groups where $\bJ$ is twisted constant and $\bG$ is reductive. Then.

\smallskip

\noindent 
{\rm (i)} Let $\widetilde E$ be a $p'-$geometric loop $\widetilde \bG$-torsor over $R_n$.
Then the twisted $R_n$-group scheme
$_{\widetilde E}\widetilde \bG$ admits a maximal torus 
which is split by a $p'-$geometric Galois cover of $R_n$.

\smallskip

\noindent  
{\rm (ii)}  
Let $\widetilde E$ be a toral $\widetilde \bG$-torsor over $R_n$
which is split by $R_n^{p'-geo}$.  Assume that
the twisted $R_n$-group scheme 
 $_{\widetilde E}\widetilde \bG$ admits a maximal torus 
which  is split by a finite Galois $p'$-extension of
$R_n$. Then $\widetilde E$ is a  $p'-$geometric loop torsor.
\end{theorem} 


\begin{proof}(i) 
Our assumption 
implies that the $\widetilde \bG$-torsor
$\widetilde E$ 
is split by a Galois  extension $R'/R_n$  whose Galois group is of the form
$\Gamma= \Gal(R'/R)=
 {\mu_m(l)^n \rtimes \Gal(l/k)}$.
Clearly, $\Gamma$ is of type (S) and the kernel of the  type map is $\mu_m(l)^n,$
 a finite commutative $p'$--group. Hence the assertion follows from 
Theorem \ref{existenceoftori} and Remark~\ref{semisimplicity}.

\smallskip

\noindent (ii) The strategy is to apply Theorem \ref{t->l} (iii).
 Let $\bT$ be a maximal $k$--torus of $\widetilde \bG$ and 
 $ \bN= N_{\widetilde \bG}(\bT)$. There exists a Galois cover  $R'/R$ as above
 which splits the $\widetilde \bG$--torsor $\widetilde E$ 
and which splits the $k$--torus $\bT$.
We have $\Pic(R')=0$.
Since all assumptions of  Theorem \ref{t->l} (iii) are
satisfied, we conclude that 
  $\widetilde E$ is a  $p'-$geometric loop torsor.
\end{proof}

To summarize, we have the following implications
for a $\widetilde \bG$-torsor $\tE$ that is split by a 
 $p'$-geometric Galois cover of $R_n$

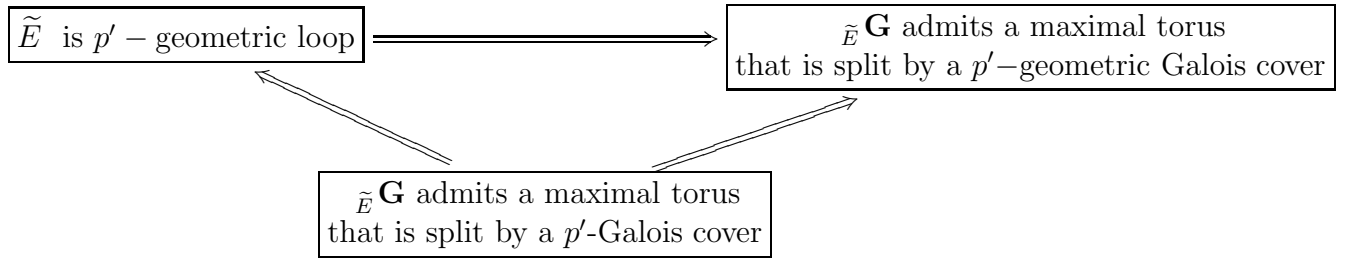
\[
\xymatrix{
\fbox{$\widetilde{E}\mbox{\ \ is\ } p'-\mbox{geometric loop}$}
\ar@{=>}[rr]
& &
\fbox{$\txt{$\,_{\tE\,}\bG$ admits a maximal torus\\ 
that is split by a\ $p'-$geometric Galois cover}$}
\\
&  
\makebox[75pt]{\fbox{$\txt{$\,_{\tE\,}\bG$ admits a maximal torus\\ 
that is split by a $p'$-Galois cover}$}} 
\ar@{=>}[ur] \ar@{=>}[ul] &
}
\]

\begin{remark} Note that the converse of the top implication fails (see the counterexamples in Remark~\ref{need}).
\end{remark}

If the Galois group of $k$ is a $p'$-profinite group,
we have a necessary and sufficient condition 
for characterizing $p'$-loop torsors.
\begin{corollary} Under the assumptions of 
Theorem \ref{thm_laurent}, assume furthermore that
$\Gal(k_s/k)$ is a profinite $p'$--group.
Let $\widetilde E$ be a $\widetilde \bG$-torsor over $R_n$
which is split by $R_n^{p'-geo}$. Then the following 
assertions are equivalent:

\smallskip

{\rm (a)} $\widetilde E$ is a  $p'$-loop torsor.

\smallskip

{\rm (b)}  The twisted $R_n$-group scheme 
 $_{\widetilde E}\widetilde{\bG}$ admits a maximal torus 
which  is split by a finite Galois $p'$-extension of
$R_n$. 
\end{corollary}

\begin{proof} The assumption 
implies that $R_n^{p'-geo}$ is a limit of Galois 
extensions of $R_n$ of $p'$-degree.
In other words,   $p'-$geometric Galois covers of $R_n$ are also $p'$-covers.
Theorem \ref{thm_laurent}(i)
yields the implication $(a) \Longrightarrow (b)$
and the converse is given by (ii).
\end{proof}

\bigskip 


\end{document}